\UseRawInputEncoding 
\documentclass[11pt,a4paper]{article}
\usepackage{amsfonts}
\usepackage{amssymb}
\usepackage{mathrsfs}
\usepackage{amsmath}
\usepackage{booktabs}
\usepackage{epsf,epsfig,amsfonts,amsgen,indentfirst}
\usepackage{amsmath,amstext,amsbsy,amsopn,amsthm,bbding,wasysym}
\usepackage{multicol,mathdots}
\usepackage{subfigure}
\allowdisplaybreaks

\newtheorem{theorem}{Theorem}[section]

\newtheorem{lemma}[theorem]{Lemma}

\newtheorem{definition}{Definition}[section]

\newtheorem{remark}{Remark}

\newtheorem{example}{Example}

\begin{document}
\title
{\LARGE \textbf{Enumeration of maximum matchings of graphs}}

\author{ Tingzeng Wu$^a$\thanks{{Corresponding author.\newline
\emph{E-mail address}: mathtzwu@163.com, zxl2748564443@163.com, lvhz@uestc.edu.cn
}}, Xiaolin Zeng$^a$, Huazhong L\"{u}$^b$\\
{\small $^{a}$School of Mathematics and Statistics, Qinghai Nationalities University, }\\
{\small  Xining, Qinghai 810007, P.R.~China} \\
{\small $^{b}$School of Mathematical Sciences,  University of Electronic Science and Technology of China, }\\
{\small  Chengdu, Sichuan 610054, P.R.~China} }
\date{}

\maketitle
\noindent {\bf Abstract:}\ \
 Counting maximum matchings in a graph is of great interest in statistical mechanics,
 solid-state chemistry, theoretical computer science, mathematics, among other disciplines. However, it is a challengeable problem to explicitly determine the number of maximum matchings of general graphs. In this paper, using Gallai-Edmonds structure theorem,   we derive a computing formula for the number of maximum matching in a graph. According to the formula, we obtain an algorithm to enumerate  maximum matchings of a graph.  In particular, The formula implies that  computing the number of  maximum matchings of a graph is converted to compute the number of perfect matchings of some induced subgraphs of the graph. As an application, we calculate the number of maximum matchings of opt trees. The result extends a conclusion obtained by Heuberger and Wagner[C. Heuberger, S. Wagner, The number of maximum matchings in a tree, Discrete Math. 311 (2011) 2512--2542].


\section{Introduction}%

Enumerating maximum matchings is a classical problem in graph theory. This problem has been intensively studied for a long time
by mathematicians and computer scientists\cite{lov}. The matching problem on a graph is equivalent to a physical model of dimers. This was mostly studied on planar graphs (lattices), where there is a beautiful method by Kasteleyn\cite{kas}, which shows how to exactly count dimer arrangements (perfect matchings).
Note that counting perfect matching of graphs  are extensively
examined, see \cite{ciu,dye, linc, wuy,xin,yanf,zhangs}  and the references therein.

Little is known about  the number of maximum matchings for graphs that do not have a perfect
matching.  Valiant\cite{val} showed that counting maximum matching of a graph  is   \#P-complete.   Henning and  Yeo\cite{hen1} derived a tight lower bound on the matching number in a graph with given maximum degree.  D\v{o}sli\'{c} and Zubac\cite{dos} calculated the maximal matchings in joins and corona products of some classes of graphs. G\'{o}rska and  Skupie\'{n}\cite{gor} found the
exponential upper and lower bounds on the maximum number of maximal matchings among  trees of order $n$.
Heuberger and  Wagner\cite{heu} improved the result by G\'{o}rska and  Skupie\'{n} on the number of maximal
matchings. And they determined all extremal trees with maximum number of maximal matching. In this paper, our purpose is to give a computational method for enumerating the maximum matchings of a connected graph.

 The rest of this
paper is organized as follows. In  Section 2, we prove the main results in this paper, and give a computing formula of maximum matching for some special graphs. In Section 3, we point out  an application of the main results that improves a  result by Heuberger and Wagner on the number of maximal matchings.

\section{The number of maximum matchings of a graph}

%

Let $G=(V(G), E(G))$ be a graph with the vertex set $V(G)=\{v_{1}, v_{2},..., v_{n}\}$ and the edge set $E(G) = \{e_{1}, e_{2},..., e_{m}\}$.  The path, cycle, star and complete graph on $n$ vertices are denoted by  $P_{n}$, $C_{n}$, $K_{1,n-1}$ and $K_{n}$, respectively.  For more notations  and terminologies not defined here, see \cite{lov}.

A {\em matching} in a graph  is a set of non-loop edges with no shared endpoints. And a {\em perfect matching} in a graph  is a matching that saturates every vertex. A {\em near-perfect matching} in a graph is a matching that only one vertex is unsaturated.
A {\em maximum matching} is a matching of maximum size among all matchings in the graph. For convenience, the number of maximum matchings of graph $G$ and the number of perfect matchings of $G$ denoted by $M_{\max}(G)$ and $M_{pm}(G)$, respectively.

Let $G$ be a graph with $n$ vertices, if $G-v$ has a perfect matching for every $v\in V(G)$, then $G$ is {\em factor-critical}. Definition \ref{art21} comes from \cite{lov}. The notation of
Definition \ref{art21} will be used throughout this paper.
\begin{definition} \label{art21}
Let $G$ be a graph. Let $D(G)$ be the set of all vertices in $G$
which are not saturated by at least one maximum matching of $G$.
Define $A(G) = \{v \in (V(G)-D(G)) :$ there exist a vertex $u \in D(G)$ with $uv \in E(G)\}$ and
 $C(G)=V(G)-(D(G) \cup A(G))$.
\end{definition}
By  Definition \ref{art21}, it can be known that $D(G)$, $A(G)$ and $C(G)$ is a vertex partition of $V(G)$. With this partition, the Gallai-Edmonds structure theorem is stated as follows.

\begin{theorem}\label{art22}{(Gallai-Edmonds Structure Theorem \cite{lov})}
Let $G$ be a
graph and let $D(G)$, $C(G)$ and $A(G)$ be the vertex-partition defined above.
Then \\
$(i)$ the components of the subgraph induced by $D(G)$ are factor-critical;\\
$(ii)$ the subgraph induced by $C(G)$ has a perfect matching;\\
$(iii)$ if $M$ is any maximum matching of $G$, it contains a near-perfect matching
of each component of $G[D(G)]$, a perfect matching of  $G[C(G)]$ and
matches all vertices of $A(G)$ with vertices in distinct components of $G[D(G)]$;\\
$(iv)$ the bipartite graph obtained from $G$ by deleting the vertices of $C(G)$ and the edges spanned by $A(G)$ and by contracting each component of $G[D(G)]$ to
a single vertex has positive surplus (as viewed from $A(G)$);\\
$(v)$ the size of maximum matching is $\frac{1}{2}(|V(G)|-c(D(G))+|A(G)|)$, where $c(D(G))$ denotes the number of components of the graph spanned by $D(G)$.
\end{theorem}
\begin{remark}\label{re1}
Let $G$ be a graph containing no  perfect matching.  By Theorem \ref{art22}, we know that a maximum matching of $G$  consists of a maximum matching in $G[D(G)]$, a perfect matching in $G[C(G)]$, and a maximum matching in  edge-induced subgraph obtained by all edges  connecting $A(G)$ to $D(G)$. This implies that every edge incident with a vertex of $D(G)$ lies in some maximum matching of $G$, and no edge induced by $A(G)$ or connecting $A(G)$ to $C(G)$ belongs to any maximum matching.
\end{remark}

\begin{theorem}\label{art23}
Let $G$ be a factor-critical graph on $n$ vertices. If $v_{i}\in V(G)$ $(i=1, 2,..., n)$, then the number of maximum matchings of $G$ equals $\sum_{i=1}^{n}M_{pm}(G-v_{i})$.
\end{theorem}

\begin{proof}
Let $G$ be a factor-critical graph on $n$ vertices, and let the vertices of $G$ be labeled by $v_{i}$ $(i=1, 2,..., n)$. For  a maximum matching in $G$, it either contains vertex $v_{1}$ or not. Thus, the number of maximum matching containing $v_{1}$ equals to the sum the number of perfect matchings of $G-v_{i}$ $(i = 2,..., n)$. And the number of maximum matching excluding $v_{1}$ equals the number of perfect matchings of $G-v_{1}$.
\end{proof}

Let $G$ be a graph with $n$ vertices. $C(G)$, $A(G)$ and $D(G)$ are defined in Definition \ref{art21}. Assume that there exist $r$ components of the subgraph induced by $D(G)$. Let $B_{i}(G)$ ($i=1,2,\ldots,r$) be the set of vertices in the $i$th component, then $D(G)=\bigcup\limits_{i=1}^{r}B_{i}(G)$. Set  the vertices of $A(G)$  as $v_{1}$, $v_{2}$, $...$, and $v_{k}$, respectively. For any $i$, $B_{i}(G)$ is contracted to a vertex $u_{i}$ in $(iv)$ of Theorem \ref{art22}. Let $H=(A,B)$ be the bipartite graph defined in Theorem \ref{art22} $(iv)$. It is easy to see that $A=\{ v_{1}, v_{2},..., v_{k}\}$ and $B=\{ u_{1}, u_{2},..., u_{r}\}$. By Theorem \ref{art22},  we obtain that every vertex of $A$ is saturated by every maximum matching of $G$.

Now we calculate the number of maximum matchings of $H$.  We define the degree of a vertex $v_{j}\in A$ $(j\in\{1,2,...,k\})$ is the number of its neighbors in $B$, denoted by $d(v_{j})$. Set  $v_{x}$ and $v_{y}$ be two vertices of $A$ in $A(G)$ such that $1 \leq x < y \leq k$. Suppose that $v_{x}$ and $v_{y}$ are neighbors of the $z$th ($1\leq z\leq r$) vertex in $B$.  The number of edges incident with the $z$th vertex of $B$ and $v_{x}$ of $A$ in $A(G)$ is denoted by $\mid e_{xy}^{z}\mid$, and the $z$th vertex of $B$ is covered by an edge that incident with $v_{y}$.

\begin{theorem}\label{art24}
Let $G$ be any graph, and let $H$ be a bipartite graph defined as above. Then
\begin{eqnarray}\label{equ23}
M_{\max}(H)=\sum_{{d(v_{k})}}\sum_{\substack{d(v_{s})-\sum\limits_{t=s+1}^{k}\mid e_{st}^{p}\mid\\s\in\{2,3,\cdots,k-1\}} }\biggl(d(v_{1})-\sum_{h=2}^{k} \mid e_{1h}^{q}\mid\biggl),
\end{eqnarray}
where $p,q\in\{ 1,2,\cdots,r\}$, the first sum ranges over all edges incident with the vertex $v_{k}$ and a vertex of $B$, the second sum ranges over all edges incident with a vertex $v_{s}$ $(s=2,3,\cdots,k-1)$ and a vertex of $B$.
\end{theorem}

\begin{proof}
Here we declare that all definitions and notations are defined the same as above. We give a method to enumerate all the maximum matchings of $H$.

By Theorem \ref{art22}, it is easy to find a maximum matching $M$ of $H$ such that every vertex $v_{i}$ and one of its neighbors in $B$ are saturated by one edge in $M$. Based on the maximum matching $M$ of $H$, we construct other maximum matchings of $H$ by the following steps.

{\bf Step 1:} On the basis of the maximum matching $M$, we replace the matching edge of $M$ incident with the vertex $v_1$ by another edge with endpoints $v_{1}$ and a vertex of $B$ not covered by $M$. If there exists no such edge, then we move to the next step. Otherwise, we change it and range over all possible edges. Hence, the number of maximum matchings of these selections is $d(v_{1})-\sum\limits_{h=2}^{k} \mid e_{1h}^{q}\mid-1$, where $q\in\{ 1,2,\cdots,r\}$.

{\bf Step 2:} Based on the previous step, we replace the matching edge of $M$ incident with the vertex $v_2$ by another edge with endpoints $v_{2}$ and a vertex in $B$ not covered by edges of $M$ incident with $v_{j'}$ ($j'=3,4,\cdots,k$). Similarly, we notice that the matching edges of $M$ incident with $v_{j'}$ remain unchanged. If there exists no such an edge satisfying the above condition, we proceed the next step. Otherwise, we further replace the matching edge incident with $v_{1}$ and a vertex of $B$ (not covered by edges of the current matching incident with $v_{j''}, j''=2,3,\cdots,k$). If there exists such an edge, then we repeat the procedure of Step 1. Otherwise, we reselect a matching edge incident with $v_{2}$ and a vertex of $B$ (not covered by edges of the current matching incident with $v_{j''}, j''=3,4,\cdots,k$). So the number of maximum matchings of these choices is $\sum\limits_{{d(v_{2})-\sum\limits_{t=3}^{k}\mid e_{2t}^{p}\mid}-1}(d(v_{1})-\sum\limits_{h=2}^{k} \mid e_{1h}^{q}\mid)$, where $p, q\in\{ 1,2,\cdots,r\}$.

{\bf Step 3:} Based on the previous step, we replace the current matching edge incident with the vertex $v_3$ by another edge with endpoints $v_{3}$ and a vertex in $B$ (not covered by edges of $M$ incident with $v_{j'''} , j'''=4,5,\cdots,k$). Note also that the matching edges of $M$ incident with $v_{j'''}$ remain unchanged.

If there exists no such an edge satisfying the above condition, we continue to replace the current matching edge incident with the vertex $v_4$ by another edge with endpoints $v_{4}$ and a vertex in $B$ (not covered by edges of $M$ incident with $v_{j''''}, j''''=5,6,\cdots,k$). Again, the matching edges of $M$ incident with $v_{j''''}$ remain unchanged.

If such an edge exists, then we replace the current matching edge incident with the vertex $v_2$ by another edge with endpoints $v_{2}$ and a vertex in $B$ (not covered by edges of $M$ incident with $v_{j'}, j'=2,3,\cdots,k$). If no such an edge exists, we need to replace the current matching edge incident with the vertex $v_3$ by another edge with endpoints $v_{3}$ and a vertex in $B$ (not covered by edges of $M$ incident with $v_{j'''}, j'''=4,5,\cdots,k$), and keep the matching edges of $M$ incident with $v_{j'''}$ remain unchanged. And we continue to repeat Step 2 if this kind of edges still exist. So the number of maximum matchings of theses selections is $\sum\limits_{{d(v_{3})-\sum\limits_{t^{'}=4}^{k}\mid e_{3t^{'}}^{o}\mid}-1}\sum\limits_{{d(v_{2})-\sum\limits_{t=3}^{k}\mid e_{2t}^{p}\mid}}(d(v_{1})-\sum\limits_{h=2}^{k} \mid e_{1h}^{q}\mid)$, where $o, p, q\in\{ 1,2,\cdots,r\}$.

Follow the previous steps by analogous reasoning until we reach the matching edge with endpoints $v_{k-1}$ and a vertex in $B$ (not covered by the edge of $M$ incident with $v_{k}$), and keep the edge incident with $v_{k}$ in $M$ unchanged. If such edges exist, then the number of maximum matchings of theses selections is $\sum\limits_{{d(v_{k-1})-\mid e_{(k-1)k}^{o}\mid}-1}\sum\limits_{\substack{d(v_{s})-\sum\limits_{t=s+1}^{k}\mid e_{st}^{p}\mid\\s\in\{2,3,\cdots,k-2\}} }(d(v_{1})-\sum\limits_{h=2}^{k} \mid e_{1h}^{q}\mid)$($o, p, q\in\{ 1,2,\cdots,r\}$). Otherwise, we should reselect a distinct matching edge incident with $v_{k}$ and a vertex of $B$, and repeat the above process. We change it and  range over all matching edges (different from the one in $M$) incident with $v_{k}$ and a vertex of $B$. Therefore, the number of maximum matchings of these choices is $\sum\limits_{{d(v_{k})}-1}\sum\limits_{\substack{d(v_{s})-\sum\limits_{t=s+1}^{k}\mid e_{st}^{p}\mid\\s\in\{2,3,\cdots,k-1\}} }(d(v_{1})-\sum\limits_{h=2}^{k} \mid e_{1h}^{q}\mid)$.

We have enumerated all maximum matchings hereto. Summing up them above, we can derive the result of (\ref{equ23}).

\end{proof}

\begin{remark}
Readers can refer to Example \ref{exa21} for a better understanding of Theorem \ref{art24}.
\end{remark}

In the following, we give a formula to calculate the number of maximum matchings of any graph $G$.

\begin{theorem}\label{art25}
Let $G$ be a graph. Then
\begin{eqnarray}\label{equ24}
M_{\max}(G)=M_{pm}(C(G)) \times\biggl[ \sum_{M_{\max}(H)} \biggl(\prod_{i_{1}=1}^{k} M_{pm}(B_{i_{1}})\times {\prod_{i_{2}=1}^{r-k}\sum_{\alpha=1}^{\beta_{i_{2}}}M_{pm}(B_{i_{2}}-v_{\alpha})}\biggl)\biggl],
\end{eqnarray}
 where the first sum ranges over all cases of the maximum matchings in $H$.
\end{theorem}

\begin{proof}
Clearly, $C(G)$ has perfect matchings and the number of perfect matchings of $C(G)$ is $M_{pm}(C(G))$. We know that each maximum matching must cover all vertices of $A(G)$. By  Theorem \ref{art22}, a maximum matching $M'$ of $H$ corresponds to $\prod\limits_{i_{1}=1}^{k}M_{pm}(B_{i_{1}})\times {\prod\limits_{i_{2}=1}^{r-k}\sum\limits_{\alpha=1}^{\beta_{i_{2}}}M_{pm}(B_{i_{2}}-v_{\alpha})}$ maximum matchings of subgraph $G'$  of $G$ induced by $A(G)$ and $D(G)$, where  $B_{i_{1}}$ is a subgraph induced by $B_{i_{1}}(G)$, and any vertex of $B_{i_{1}}$ is  matched by $M'$, $B_{i_{2}}$ is a subgraph induced by $B_{i_{2}}(G)$, no vertex of $B_{i_{2}}(G)$ is covered by $M'$, and the number of vertices in $B_{i_{2}}(G)$ is $\beta_{i_{2}}$. Thus $ M_{\max}(G^{'})=\sum\limits_{M_{\max}(H)} \biggl(\prod\limits_{i_{1}=1}^{k} M_{pm}(B_{i_{1}})\times {\prod\limits_{i_{2}=1}^{r-k}\sum\limits_{\alpha=1}^{\beta_{i_{2}}}M_{pm}(B_{i_{2}}-v_{\alpha})}\biggl)$, where the first sum ranges over all cases of the maximum matchings in $H$. By  Theorem \ref{art22}$(iii)$, we can get the number of maximum matchings of $G$, and  the formula (\ref{equ24}) follows.

Thus the theorem has been proved.
\end{proof}
\begin{example}\label{exa21}
 Calculate the number of maximum matchings of the graph $G$ in Figure \ref{fig5} below.
\begin{figure}[htbp]
\begin{center}
\includegraphics[scale=0.5]{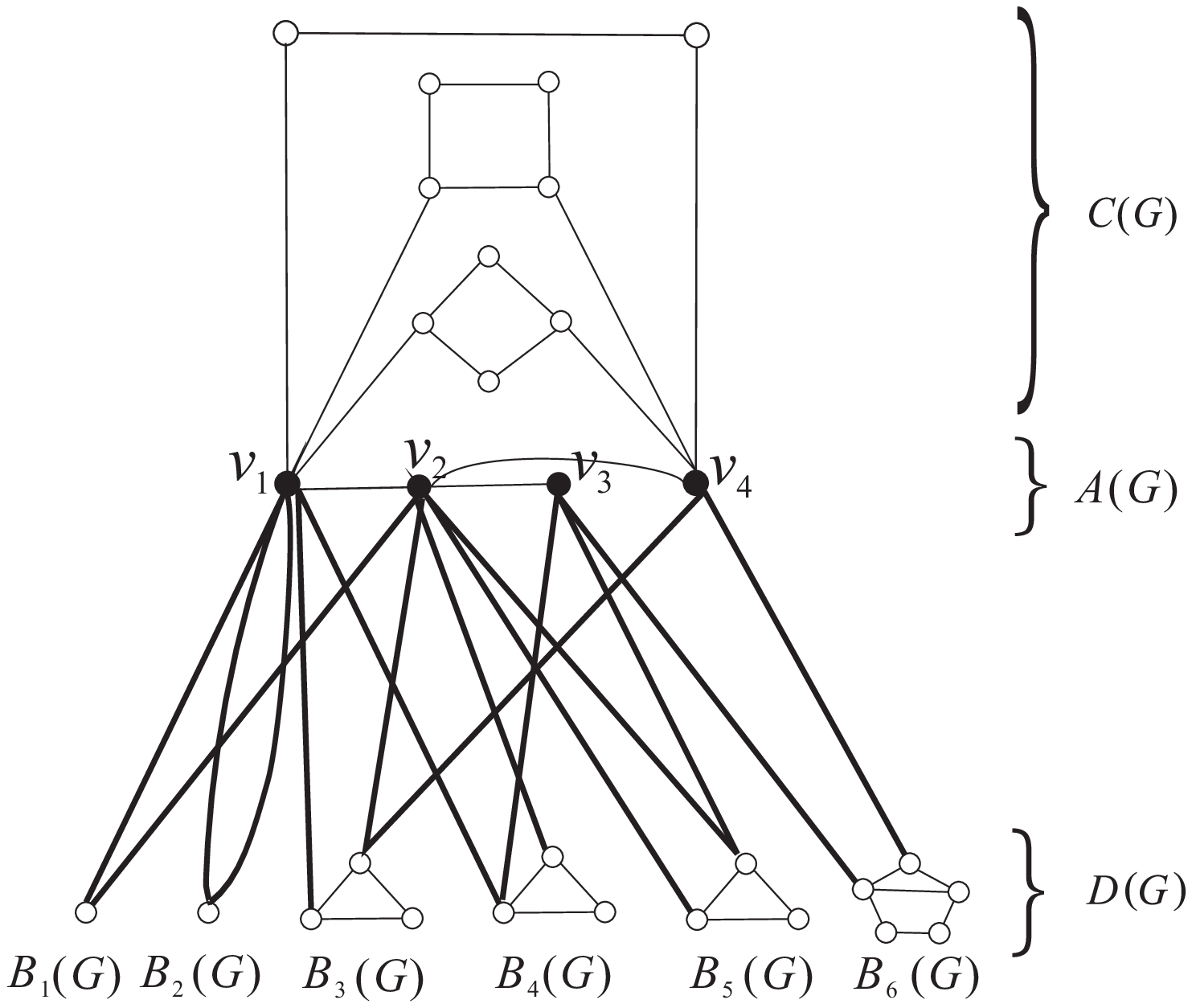}
\caption{\label{fig5}\small
{Graph $G$.}}
\end{center}
\end{figure}
\end{example}
By Theorem \ref{art22}, it is easy to obtain  a vertex partition of $V(G)=C(G)\cup A(G)\cup D(G)$, see Figure \ref{fig5}.
Let the $i$th component of $D(G)$ be $B_{i}(G)(i\in\{1,2,3,4,5,6\})$, and let $A(G)=\{v_{1},v_{2},v_{3},v_{4}\}$.
In addition, let $m_{ji}$ be the edge with endpoints $v_{j}$ and a vertex in $B_{i}(G)$, where $j\in\{1,2,3,4\}$ and $i\in\{1,2,3,4,5,6\}$.

Obviously, $C(G)$ has perfect matchings and $M_{pm}(C(G))=4$. In the following, we enumerate all maximum matchings of $H$.

By Theorem \ref{art24}, we can choose a maximum matching $M=\{m_{12},m_{21},m_{31}, m_{41}\}$ in $H$.

Step 1. Based on the $M$, we replace $m_{12}$ by another edge in $H$ incident with $v_{1}$ and a vertex of $B_{i}(i\in\{1,2,3,4,5,6\})$, and the vertex is not covered by $M$. By formula (\ref{equ23}), we get the number of maximum matchings  selected is $d(v_{1})-\mid e_{12}^{1}\mid-\mid e_{13}^{4}\mid-\mid e_{14}^{3}\mid-1=1$.

Step 2. On the basis of Step 1, we replace $m_{21}$ by an edge in $H$ incident with $v_{2}$ and a vertex of $B_{i}$, which is not covered by edges of $M$ incident with $v_{3}$ and $v_{4}$. By calculating, we have $d(v_{2})-\mid e_{23}^{4}\mid-\mid e_{24}^{3}\mid-1=2$, i.e. we can choose the edges $m_{24}$ or $m_{25}$ incident with $v_{2}$. Then we further select a matching edge incident with $v_{1}$. Hence,
\begin{eqnarray*}
\sum\limits_{{d(v_{2})-\sum\limits_{t=3}^{4}\mid e_{2t}^{p}\mid}-1}(d(v_{1})-\sum\limits_{h=2}^{4} \mid e_{1h}^{q}\mid)&=&\sum\limits_{2}(d(v_{1})-\sum\limits_{h=2}^{4} \mid e_{1h}^{q}\mid)\\
&=&(d(v_{1})- \mid e_{13}^{4}\mid-\mid e_{14}^{3}\mid)+(d(v_{1})- \mid e_{13}^{4}\mid-\mid e_{14}^{3}\mid)\\
&=&3+3=6.
\end{eqnarray*}

Step 3. On the basis of Step 2, we replace the edge incident with $v_{3}$ in $M$ by an edge incident with $v_{3}$ and a vertex of $B_{i}$ that is not covered by the edge of $M$ incident with $v_{4}$. By calculating, we have $d(v_{3})-1=2$, i.e. we can also select $m_{32}$ or $m_{33}$. Then we choose a matching edge in $H$ incident with $v_{2}$ and a vertex of $B_{i}$, so there are $(d(v_{2})-\mid e_{23}^{5}\mid-\mid e_{24}^{3}\mid)+(d(v_{2})-\mid e_{24}^{3})=2+4=6$ choices.

That is, if we select the matching edge $m_{32}$, then we can further choose $m_{21}$ or $m_{23}$ incident with $v_{2}$. If we select $m_{33}$, then we can further choose $m_{21}$, $m_{23}$, $m_{24}$, $m_{25}$ incident with $v_{2}$. Now, we are ready to choose the matching edge incident with $v_{1}$. By formula (\ref{equ23}), we get the number of maximum matchings that are selected is
\begin{eqnarray*}
\sum\limits_{d(v_{3})-1}\sum\limits_{{d(v_{2})-\sum\limits_{t=3}^{4}\mid e_{2t}^{p}\mid}}(d(v_{1})-\sum\limits_{h=2}^{4} \mid e_{1h}^{q}\mid)&=&\sum\limits_{2}\sum\limits_{{d(v_{2})-\sum\limits_{t=3}^{4}\mid e_{2t}^{p}\mid}}(d(v_{1})-\sum\limits_{h=2}^{4} \mid e_{1h}^{q}\mid)\\
&=&(d(v_{1})-\mid e_{12}^{1}\mid- \mid e_{14}^{3}\mid)+(d(v_{1})- \mid e_{12}^{4}\mid-\mid e_{14}^{3}\mid)\\
&&+(d(v_{1})- \mid e_{12}^{1}\mid-\mid e_{14}^{3}\mid)+(d(v_{1})- \mid e_{12}^{4}\mid-\mid e_{14}^{3}\mid)\\
&&+(d(v_{1})-\mid e_{14}^{3}\mid)+(d(v_{1})-\mid e_{14}^{3}\mid)\\
&=&3+3+3+3+4+4=20.
\end{eqnarray*}

Next reselect a matching edge incident with $v_{4}$ and repeat the above process. Since $d(v_{4})-1=1$, we have exactly one choice to select $m_{42}$ incident with $v_{4}$. Thus, we have a new initial maximum matching $M^{'}=\{m_{12}, m_{21}, m_{31}, m_{42}\}$ in $H$.

Step 1. On the basis of $M^{'}$, we replace an edge incident with $v_{1}$ in $M^{'}$ by a matching edge in $H$ incident with $v_{1}$ and a vertex of $B_{i}$ not covered by $M^{'}$. Then the number of maximum matchings of the selections is $d(v_{1})-\mid e_{12}^{1}\mid-\mid e_{13}^{4}\mid-1=2$.

Step 2. Based on the the Step 1, we replace an edge incident with $v_{2}$ in $M^{'}$ by a matching edge in $H$ incident with $v_{2}$ and a vertex of $B_{i}$ not covered by edges incident with $v_{3}$ and $v_{4}$ in $M^{'}$.
By calculating, we have $d(v_{2})-\mid e_{23}^{4}\mid-1=3$. It means that we can choose edges $m_{22}$, $m_{24}$ or $m_{25}$ incident with $v_{2}$. Then we further select a matching edge incident with $v_{1}$. By formula (\ref{equ23}), we have
\begin{eqnarray*}
\sum\limits_{{d(v_{2})-\sum\limits_{t=3}^{4}\mid e_{2t}^{p}\mid}-1}(d(v_{1})-\sum\limits_{h=2}^{4} \mid e_{1h}^{q}\mid)&=&\sum\limits_{3}(d(v_{1})-\sum\limits_{h=2}^{k} \mid e_{1h}^{q}\mid)\\
&=&(d(v_{1})- \mid e_{12}^{3}\mid-\mid e_{13}^{4}\mid)+(d(v_{1})-\mid e_{13}^{4}\mid)+\\
&&(d(v_{1})-\mid e_{13}^{4}\mid)\\
&=&3+4+4=11.
\end{eqnarray*}

Step 3. On the basis of Step 2, we replace a matching edge in $H$ incident with $v_{3}$ and a vertex of $B_{i}$ not covered by the edge of $M^{'}$ incident with $v_{4}$. So we have $d(v_{3})-\mid e_{34}^{6}\mid-1=1$ choice, i.e. we can select the edge $m_{32}$ at this time.

Then we further choose a matching edge in $H$ incident with $v_{2}$ and a vertex of $B_{i}$. By calculating, we have $d(v_{2})-\mid e_{23}^{5}\mid=3$ choices, i.e. we can select the edges $m_{21}$, $m_{22}$ or $m_{23}$. Finally, we choose a matching edge in $H$ incident with $v_{1}$ and a vertex of $B_{i}$. By formula (\ref{equ23}), we have
\begin{eqnarray*}
\sum\limits_{d(v_{3})-\mid e_{34}^{6}\mid-1}\sum\limits_{{d(v_{2})-\sum\limits_{t=3}^{4}\mid e_{2t}^{p}\mid}}(d(v_{1})-\sum\limits_{h=2}^{4} \mid e_{1h}^{q}\mid)&=&\sum\limits_{3}(d(v_{1})-\sum\limits_{h=2}^{4} \mid e_{1h}^{q}\mid)\\
&=&(d(v_{1})-\mid e_{12}^{1}\mid )+(d(v_{1})- \mid e_{12}^{3}\mid)\\
&&+(d(v_{1})- \mid e_{13}^{4}\mid)\\
&=&4+4+4=12.
\end{eqnarray*}

Summing up all the maximum matching enumerated above, it follows that the number of maximum matchings in $H$ is 54.

Next, we calculate the number of maximum matchings of the subgraph $G'$  of $G$ induced by $A(G)$ and $D(G)$. By Theorem  \ref{art25} and the selection of the maximum matching in $H$, we have
  \begin{eqnarray*}
M_{\max}(G')
&=& \sum\limits_{54} (\prod\limits_{i_{1}=1}^{4} M_{pm}(B_{i_{1}})\times {\prod_{i_{2}=1}^{2}\sum\limits_{\alpha=1}^{\beta_{i_{2}}}M_{pm}(B_{i_{2}}-v_{\alpha})})\\
&=&370,
 \end{eqnarray*}
where $\beta_{i_{2}}\in\{1,3,5\}$, the first sum ranges over all maximum matchings in $H$. Therefore, the number of maximum matchings of $G$ is $M_{\max}(G)=4 \times370=1480$.

Base on arguments as above, we present an algorithm to calculate the number of maximum matchings of $G$.

\noindent\rule{\textwidth}{0.5pt}
{\bf Algorithm } Calculating the number of maximum matchings of $G$.

{\bf Step 1: }If the number of vertices in graph $G$ is 0 or 1, then output $M_ {max}(G)=1$, stop.

{\bf Step 2: }According to Edmonds Blossom Algorithm \cite{lov}, find a maximum matching $M$. Using $M$,  we can obtain the vertex partition $C(G)$, $A(G)$ and $D(G)$ of $V(G)$.

{\bf Step 3: } By Formula (\ref{equ24}), we get $M_{\max}(G)$. Output $M_{\max}(G)$.

\noindent\rule{\textwidth}{0.5pt}
\begin{remark}
By Theorem \ref{art25}, we know that computing the number of maximum matching of a graph $G$ is converted to compute $M_{pm}(G[C(G)])$ and $M_{pm}(G[B_{i}(G)-v])$.
\end{remark}

\section{An application}

Heuberger and  Wagner\cite{heu} characterised that if there is a tree on $n\geq 4$ and $n\notin\{6,10,13,20,\\34\}$, then it has a tree $T_{n}^{*}$ that it maximises $M_{\max}(T)$ over all trees of the same order. And they characterized the structure of these trees. A natural problem is how to compute exact values of the number of maximum matchings of these trees? We will give the solution of the problem in this section. For convenience, we use the same definitions and symbols as Ref.\cite{heu}. Heuberger and  Wagner defined the induced subgraph $L$(leaf), $F$, $C^{k_{j}}$, $C_{*}^{k_{j}}$, $C^{k_{j}}L$ and $ C^{k_{j}}F(k\geq1, j\in N)$ from $T_{n}^{*}$, see Figure \ref{fig2}(a). Based on these symbols,   we also need to define some new symbols as follows. Let $G_{1}$ and $G_{2}$ be vertex-disjoint graphs, and graph $G_{1}G_{2}$ obtained from $G_{1}$ and $G_{2}$ by identifying a vertex $u$ of $G_{1}$ with a vertex $v$ of $G_{2}$. Then

(1) $C_{k_{j}}P_{3}$ is derived from $C_{k_{j}}$ and $P_{3}$ by identifying a vertex $u$ of $C_{k_{j}}$ with a vertex $v$ of $P_{3}$;

(2) $C_{k_{j}}F-L$ is obtained from $C_{k_{j}}F$ which omit a vertex $v$;

(3) $C_{k_{j}}FL$ is gained by identifying a vertex $u$ of $C_{k_{j}}F$ with  $L$. \\
where $k\geq1$ and $j\in\{1,2,. . . \}$, see Figure \ref{fig2}(b).
\begin{figure}[htbp]
   \centering
   \subfigure[]{
       \begin{minipage}[t]{0.5\linewidth}
          \centering
          \includegraphics[scale=0.5]{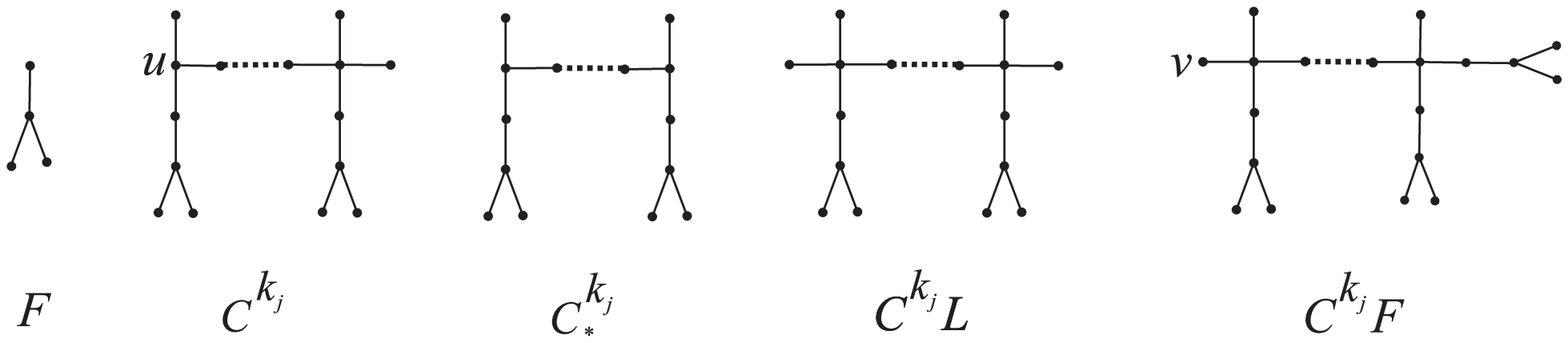}\\

     \end{minipage}
     }

     \subfigure[]{
       \begin{minipage}[t]{0.5\linewidth}
          \centering
          \includegraphics[scale=0.5]{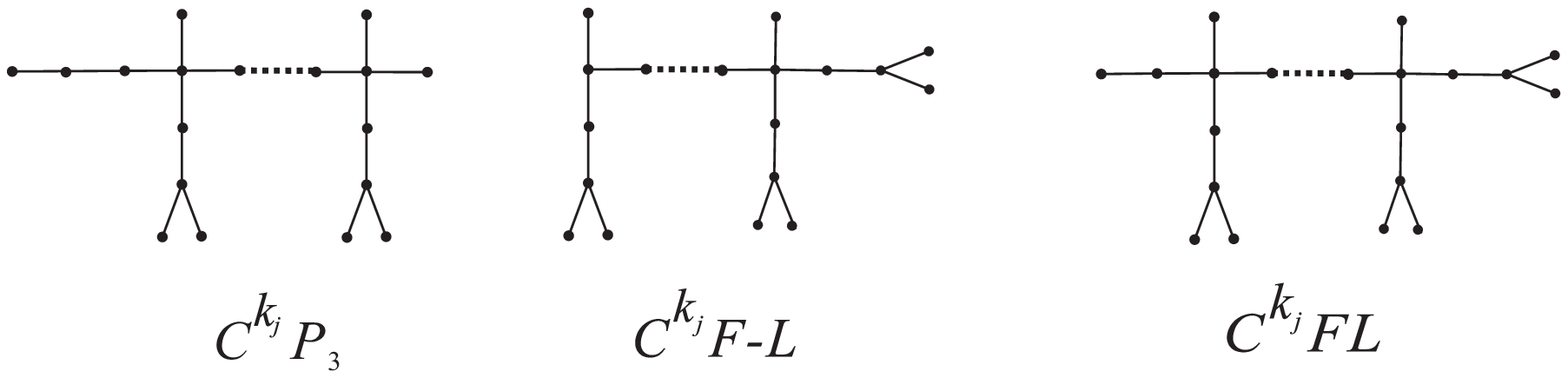}\\
          \vspace{0.02cm}

     \end{minipage}
     }
     \centering
     \caption{The structure of some induced graphs of $T_{n}^{*}$.}
     \vspace{-0.2cm}
     \label{fig2}
\end{figure}

Heuberger and  Wagner\cite{heu} discussed  which tree has the number of maximum matchings in all trees. And they obtain the following result.
\begin{lemma}\label{art31}\cite{heu}
 Let $n\geq 4$ and $n\notin\{6,10,13,20,34\}$. There is a tree $T_{n}^{*}$ of order $n$.

 (1) If $n\equiv1(mod~7)$, then $T_{n}^{*}=C^{(n-1)/7}L$.

 (2) If $n\equiv2(mod~7)$, then $T_{n}^{*}$ is shown in Figure \ref{fig3}(a) and (b), where
 \begin{equation*}
k_{0}=max\{0, \lfloor\frac{n-37}{35}\rfloor\},
k_{j}=
\begin{cases}
\lfloor\frac{n-2+7j}{35}\rfloor & \text{if $n\geq 37$};\\
\lfloor\frac{n-9+7j}{35}\rfloor & \text{if $n\leq 30$}.
\end{cases}
 \end{equation*}
and $j\in \{1,2,3,4\}$.

(3) If $n\equiv3(mod~7)$, then $T_{n}^{*}$ is shown in Figure \ref{fig3}(c), where $k_{j}=\lfloor \frac{n-17+7j}{28}\rfloor, j\in\{0,1,2,3\}$.

(4) If $n\equiv4(mod~7)$, then $T_{n}^{*}=C^{(n-4)/7}F$.

(5) If $n\equiv5(mod~7)$, then $T_{n}^{*}$ is shown in Figure \ref{fig3}(e), where  $k_{j}=\lfloor \frac{n-5+7j}{21}\rfloor, j\in\{0,1,2\}$.

(6) If $n\equiv6(mod~7)$, then $T_{n}^{*}$ is shown in Figure \ref{fig3}(f), where $k_{j}=\lfloor\frac{n-27+7j}{49}\rfloor, 0\leq j\leq 6$.

(7) If $n\equiv0(mod~7)$, then $T_{n}^{*}$ is shown in Figure \ref{fig3}(d), where $k=\frac{n-7}{7}$.
\end{lemma}

On the basis of the outline graph of trees $T_{n}^{*}$, we give specific structures of $T_{n}^{*}$ when $n\equiv 0(mod~7)$, $n\equiv 2(mod~7)$ and $n\geq 9$, $n\equiv 3(mod~7)$ and $n\geq 17$, $n\equiv 5(mod~7)$ and $n\geq 12$, $n\equiv 6(mod~7)$ and $n\geq 27, n\neq 34$. See Figure \ref{fig3}.
\begin{figure}[htbp]
   \centering
   \subfigure[$n\equiv2(mod7)$, $37\leq n\leq65$]{
       \begin{minipage}[t]{0.28\linewidth}
          \centering
          \includegraphics[scale=0.41]{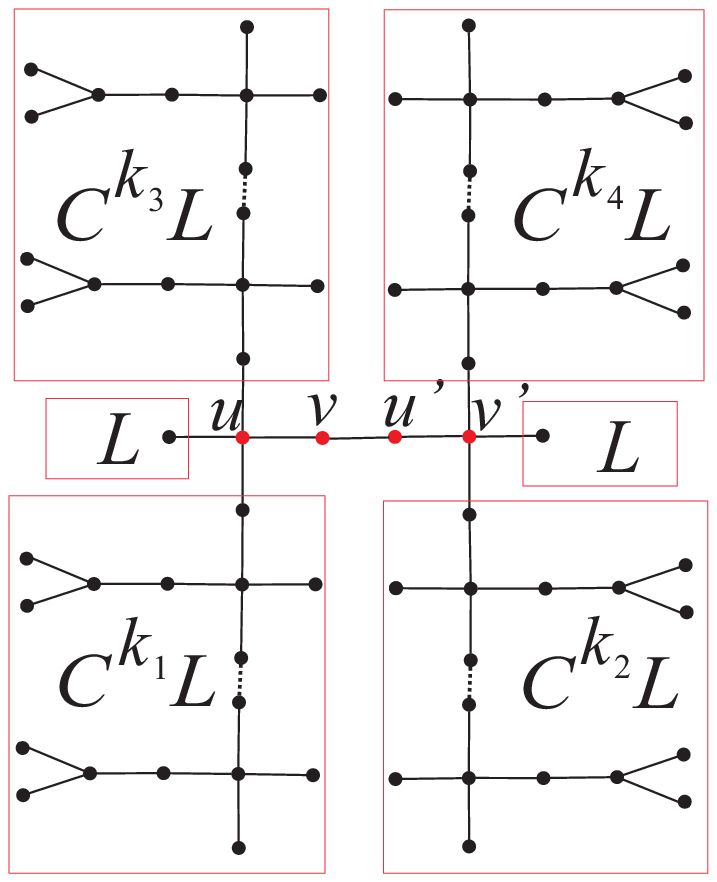}\\

     \end{minipage}
     }
     \subfigure[$n\equiv2(mod7)$ and $n\geq 65$]{
       \begin{minipage}[t]{0.28\linewidth}
          \centering
          \includegraphics[scale=0.41]{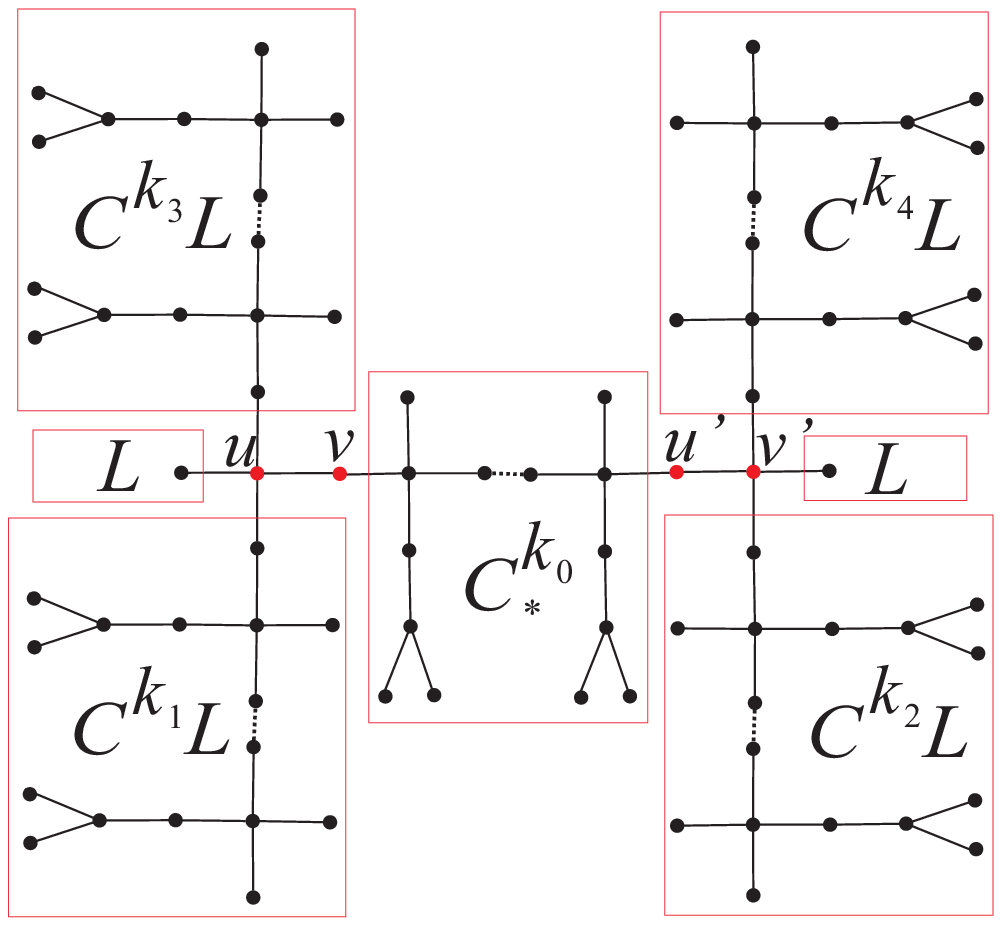}\\

     \end{minipage}
     }
     \subfigure[$n\equiv3(mod7)$ and $n\geq17$]{
       \begin{minipage}[t]{0.3\linewidth}
          \centering
          \includegraphics[scale=0.41]{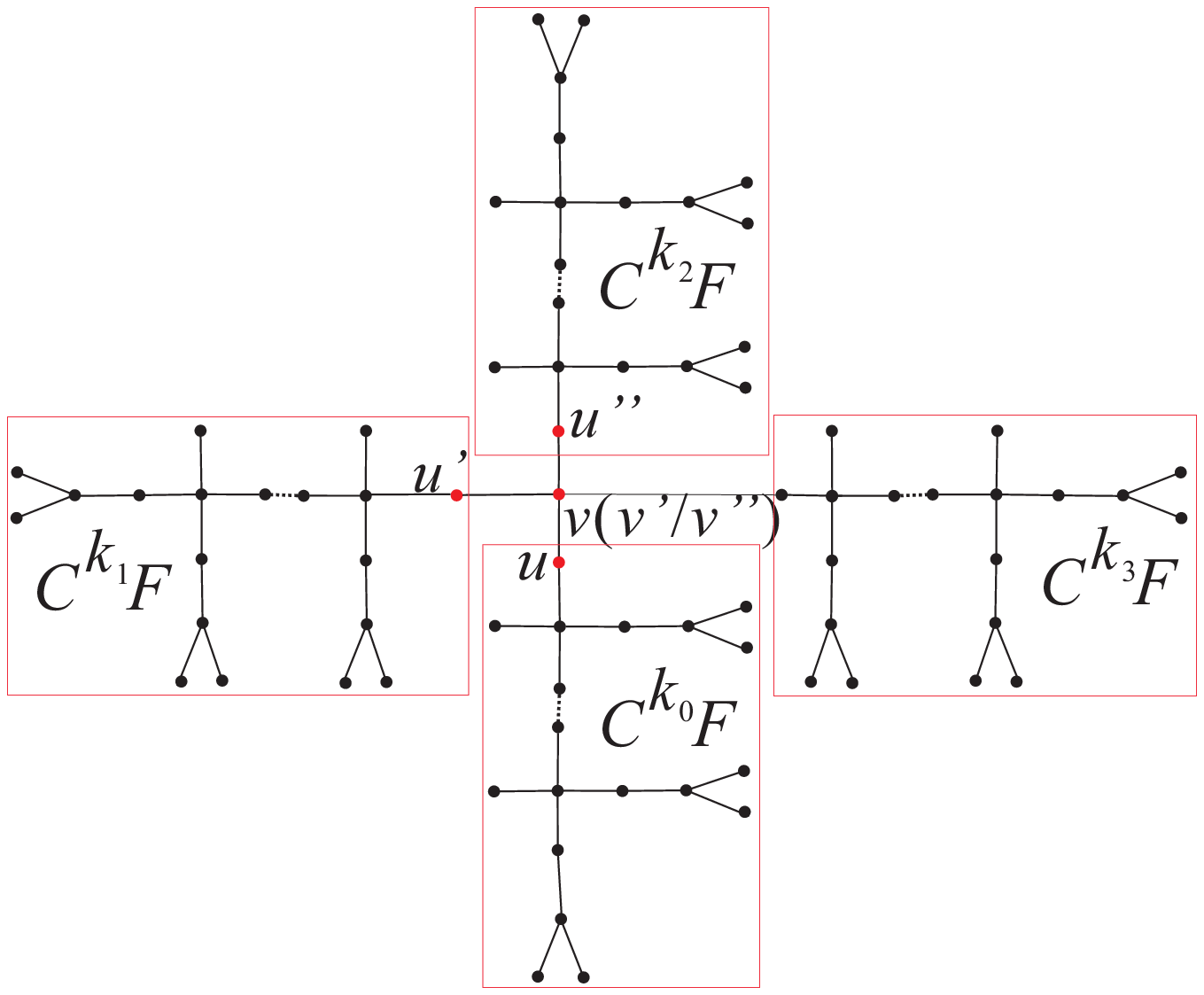}\\

     \end{minipage}
     }

 \subfigure[$n\equiv0(mod7)$]{
       \begin{minipage}[t]{0.2\linewidth}
          \centering
          \includegraphics[scale=0.44]{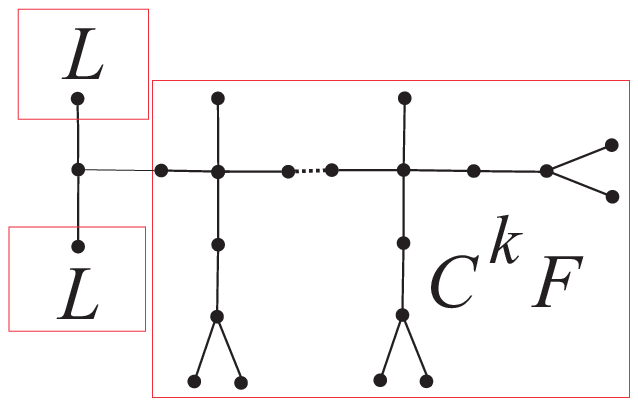}\\

     \end{minipage}
     }
   \subfigure[$n\equiv5(mod7)$]{
       \begin{minipage}[t]{0.23\linewidth}
          \centering
          \includegraphics[scale=0.41]{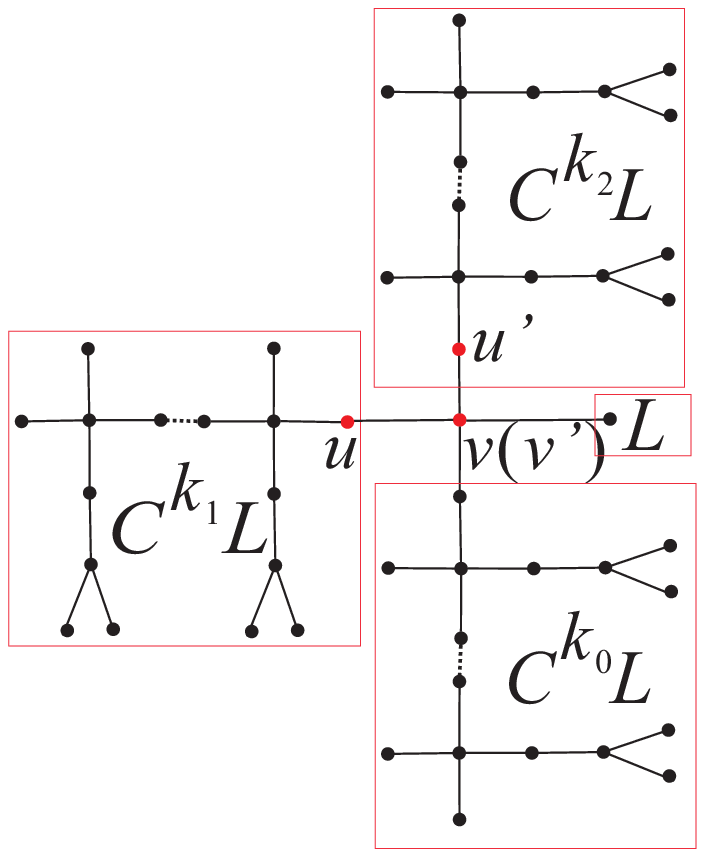}\\

     \end{minipage}
     }
     \subfigure[$n\equiv6(mod7)$, $n\geq 27, n\neq 34$]{
       \begin{minipage}[t]{0.43\linewidth}
          \centering
          \includegraphics[scale=0.41]{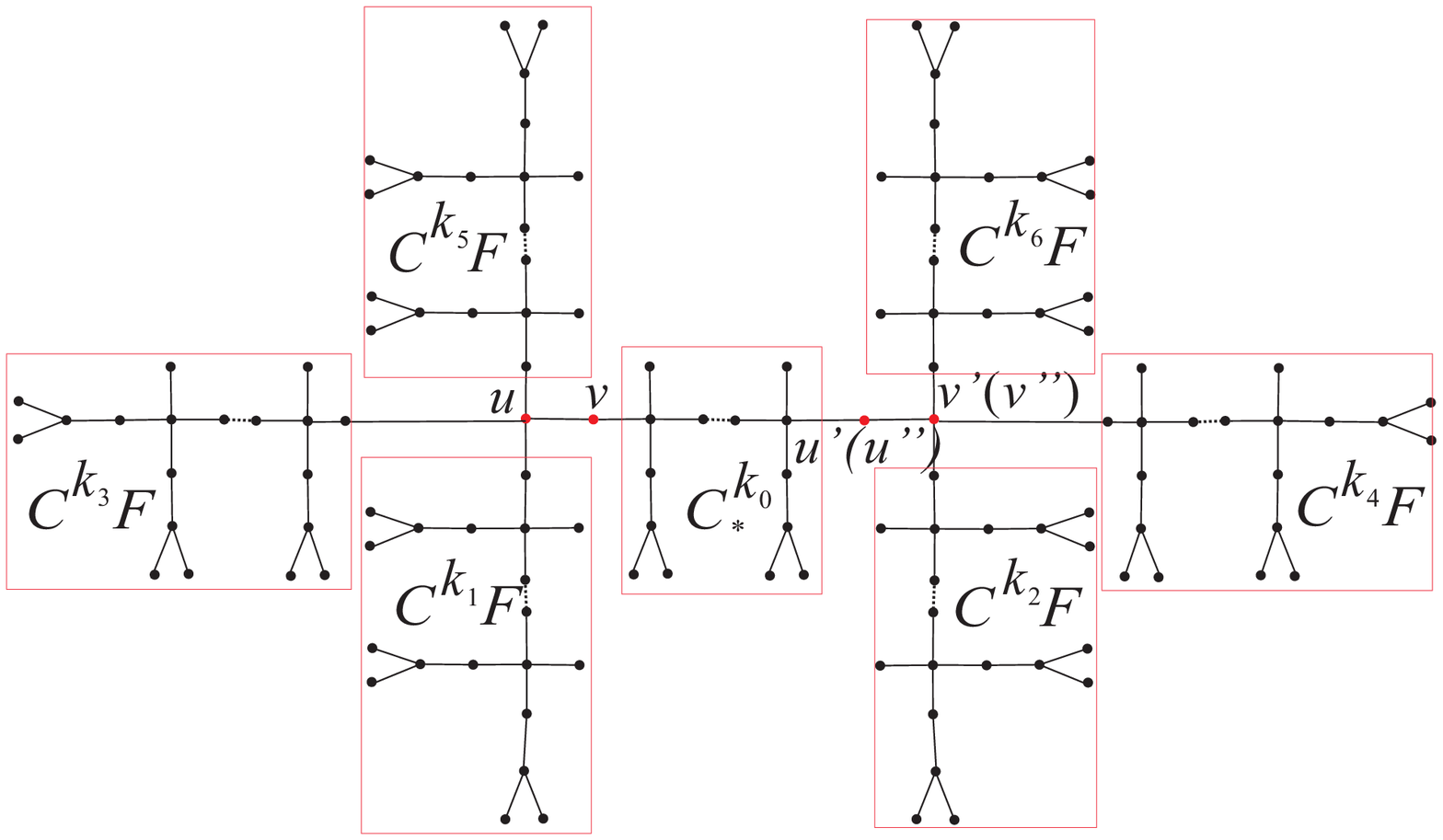}\\

     \end{minipage}
     }

     \centering
     \caption{The structure of some trees $T_{n}^{*}$.}
     \vspace{-0.2cm}
     \label{fig3}
\end{figure}

\begin{lemma}\label{art32}\cite{lov}
Let $G$ be a graph and let $\Phi_{k}(G)$ denote the number of  $k$ matchings in $G$, suppose $uv\in E(G)$. Then
\begin{eqnarray}\label{equ31}
\Phi_{k}(G)=\Phi_{k}(G-uv)+\Phi_{k-1}(G-u-v).
\end{eqnarray}
\end{lemma}


Before calculating the number of maximum matchings of $T_{n}^{*}$, we calculate the number of maximum matchings of those induced subgraph by the formula (\ref{equ24}).

\begin{lemma}\label{art33} The induced subgraphs $C^{k_{j}}L$, $C^{k_{j}}$, $C^{k_{j}}P_{3}$, $C^{k_{0}}_{*}$, $C^{k_{j}}F$, $C^{k_{j}}FL$, $C^{k_{j}}F-1$ obtained by trees $T_{n}^{*}$ with $n\geq 4$ and $n\notin\{6,10,13,20,34\}$are defined as above. Then by  formula (\ref{equ24}) we have
\begin{eqnarray*}
(i)~~~~~M_{\max}(C^{k_{j}}L)=11M_{\max}(C^{k_{j}-1}L)-9M_{\max}(C^{k_{j}-2}L),
\end{eqnarray*}
where $ k_{j}\geq3, 0\leq j\leq 6$. The initial conditions are $M_{\max}(C^{1}L)=11, M_{\max}(C^{2}L)=112$.

(ii) By the recurrence of (i),
\begin{eqnarray*}
M_{\max}(C^{k_{j}})=
5M_{\max}(C^{k_{j}-1}L)+3M_{\max}(C^{k_{j}-1}).
\end{eqnarray*}
for $k_{j}\geq 2, 0\leq j\leq 6$ with the initial condition $M_{\max}(C^{1})=8$.

(iii) By the recurrence of (i) and (ii),
\begin{eqnarray*}
M_{\max}(C^{k_{j}} P_{3})=13M_{\max}(C^{k_{j}}L)+6M_{\max}(C^{k_{j}}),
\end{eqnarray*}
where $k_{j}\geq 2, 0\leq j\leq 6$, and the initial condition is $M_{\max}(C^{1}P_{3})=19$.

(iv) By the recurrence of (ii),
\begin{eqnarray*}
M_{\max}(C^{k_{0}}_{*})=5M_{\max}(C^{k_{0}-1})+3M_{\max}(C^{k_{0}-1}_{*}),
\end{eqnarray*}
for $k_{0}\geq 2$ with the initial condition $M_{\max}(C^{1}_{*})=5$.

(vi) By the recurrence of (i) and (ii),
\begin{eqnarray*}
M_{\max}(C^{k_{j}}F)=3M_{\max}(C^{k_{j}})+6M_{\max}(C^{k_{j}-1}L),
\end{eqnarray*}
where $k_{j}\geq 2, 0\leq j\leq 6$ with the initial condition $M_{\max}(C^{1}F)=30$.

(vii) By the recurrence of (vi),
\begin{eqnarray*}
M_{\max}(C^{k_{j}}F-L)=5M_{\max}(C^{k_{j}-1}F)+3M_{\max}(C^{k_{j}-1}F-L),
\end{eqnarray*}
for $k_{j}\geq 2, 0\leq j\leq 6$ and the initial condition is $M_{\max}(C^{1}F-L)=21$.

(viii) By the recurrence of (vi),
\begin{eqnarray*}
M_{\max}(C^{k_{j}}FL)=5M_{\max}(C^{k_{j}-1}F)+3M_{\max}(C^{k_{j}-1}FL),
\end{eqnarray*}
for $k_{j}\geq 2, 0\leq j\leq 6$ and the initial condition is $M_{\max}(C^{1}FL)=21$.
\end{lemma}

\begin{theorem}\label{art34}
Let $T_{n}^{*}$ be a tree with $n(n\geq4, n\notin\{6,10,13,20,34\})$,

(1) If $n\equiv 0(mod~7)$ and $n\geq14$, then
\begin{eqnarray*}
M_{\max}(T_{n}^{*})=3M_{\max}(C^{k}F-L)+6M_{\max}(C^{k}F),
\end{eqnarray*}
where $k=\frac{n-7}{7}$, and the initial condition is $M_{\max}(T_{7}^{*})=8$.

(2) If $n\equiv 1(mod~7)$ and $n\geq 22$, then
\begin{eqnarray*}
M_{\max}(T_{n}^{*})=M_{\max}(C^{(n-1)/7}L)=11M_{\max}(C^{(n-8)/7}L)-9M_{\max}(C^{(n-15)/7}L),
\end{eqnarray*}
 the initial conditions are $M_{\max}(C^{1}L)=11, M_{\max}(C^{2}L)=112$.

(3) If $n\equiv 4(mod~7)$ and $n\geq 18$, then
\begin{eqnarray*}
M_{\max}(T_{n}^{*})=M_{\max}(C^{(n-4)/7}F)=3M_{\max}(C^{(n-4)/7})+6M_{\max}(C^{(n-11)/7}L),
 \end{eqnarray*}
the initial condition is $M_{\max}(C^{1}F)=30$.

(4) If $n\equiv 3(mod~7)$ and $n\geq 17$, then
\begin{equation*}
M_{\max}(T_{n}^{*})=
\begin{cases}
216 & \text{if $n=17$};\\
2187 & \text{if $n=24$};\\
22140 & \text{if $n=31$};\\
224100 & \text{if $n=38$};\\
\prod ^{2}_{j=0}(3M_{\max}(C^{k_{j}})+6M_{\max}(C^{k_{j}-1}F))(5M_{\max}(C^{k_{3}-1}F))+\\3M_{\max}(C^{k_{3}-1}FL)+\prod _{j=0,1,3}(3M_{\max}(C^{k_{j}})+6M_{\max}(C^{k_{j}-1}F)\\\times(5M_{\max}(C^{k_{2}-1}F)+3M_{\max}(C^{k_{2}-1}F-L))
+\prod _{j=0,2,3}(3M_{\max}(C^{k_{j}})\\+6M_{\max}(C^{k_{j}-1}F)(5M_{\max}(C^{k_{1}-1}F))+3M_{\max}(C^{k_{1}-1}F-L)\\+\prod ^{3}_{j=1}(3M_{\max}(C^{k_{j}})+6M_{\max}(C^{k_{j}-1}F)(5M_{\max}(C^{k_{0}-1}F)\\+3M_{\max}(C^{k_{0}-1}F-L)) & \text{if $n\geq45$}.
\end{cases}
 \end{equation*}
 where $k_{j}=\lfloor \frac{n-17+7j}{28}\rfloor, j\in\{0,1,2,3\}$.

(5) If $n\equiv 5(mod~7)$ and $n\geq 12$, then
\begin{equation*}
M_{\max}(T_{n}^{*})=
\begin{cases}
41 & \text{if $n=12$};\\
418 & \text{if $n=19$};\\
\prod^{2}_{j=1}(11M_{\max}(C^{k_{j}-1}L)-9M_{\max}(C^{k_{j}-2}L)) (13M_{\max}(C^{k_{0}}L)\\+6M_{\max}(C^{k_{0}}))+\prod^{1}_{j=0}(11M_{\max}(C^{k_{j}-1}L)-9M_{\max}(C^{k_{j}-2}L)) \\ \times(5M_{\max}(C^{k_{2}-1}L)+3M_{\max}(C^{k_{2}-1}))+\prod_{j=0, 2}(11M_{\max}(C^{k_{j}-1}L)\\-9M_{\max}(C^{k_{j}-2}L))(5M_{\max}(C^{k_{1}-1}L)+3M_{\max}(C^{k_{1}-1})) & \text{if $n\geq 26$}.
\end{cases}
 \end{equation*}
where $k_{j}=\lfloor \frac{n-5+7j}{21}\rfloor, j\in\{0,1,2\}$.

(6) If $n\equiv 2(mod~7)$ and $n\geq 9$ , then

(i) $M_{\max}(T_{9}^{*})=15$, $M_{\max}(T_{16}^{*})=153$, $M_{\max}(T_{23}^{*})=1560$, $M_{\max}(T_{30}^{*})=15807$.

(ii) When $37\leq n \leq 65$,
\begin{eqnarray}\label{equ33}
  &&M_{\max}(T_{n}^{*})\nonumber\\
  &=&\prod^{2}_{j=1}(13M_{\max}(C^{k_{j}}L)+6M_{\max}(C^{k_{j}}))\prod^{4}_{j=3}(11M_{\max}(C^{k_{j}-1}L)-9M_{\max}(C^{k_{j}-2}L))\nonumber\\
 &&+(13M_{\max}(C^{k_{1}}L)+6M_{\max}(C^{k_{1}}))[\prod^{3}_{j=2}(11M_{\max}(C^{k_{j}-1}L)-9M_{\max}(C^{k_{j}-2}L))\nonumber\\
 &&\times(5M_{\max}(C^{k_{4}-1}L)+3M_{\max}(C^{k_{4}-1}))+\prod^{4}_{j=2}(11M_{\max}(C^{k_{j}-1}L)-9M_{\max}(C^{k_{j}-2}L))] \nonumber\\
 && +(13M_{\max}(C^{k_{2}}L)+6M_{\max}(C^{k_{2}}))[\prod_{j=1,4}(11M_{\max}(C^{k_{j}-1}L)-9M_{\max}(C^{k_{j}-2}L))\\
&&\times(5M_{\max}(C^{k_{3}-1}L)+3M_{\max}(C^{k_{3}-1}))+\prod_{j=1,3,4}(11M_{\max}(C^{k_{j}-1}L)-9M_{\max}(C^{k_{j}-2}L))]\nonumber\\
&&+\prod^{2}_{j=1}(11M_{\max}(C^{k_{j}-1}L)-9M_{\max}(C^{k_{j}-2}L))\prod^{4}_{j=3}(5M_{\max}(C^{k_{j}-1}L)+3M_{\max}(C^{k_{j}-1}))\nonumber\\
&&+\prod_{j=1,2,4}(11M_{\max}(C^{k_{j}-1}L)-9M_{\max}(C^{k_{j}-2}L))(5M_{\max}(C^{k_{3}-1}L)+3M_{\max}(C^{k_{3}-1}))\nonumber\\
&&+\prod^{3}_{j=1}(11M_{\max}(C^{k_{j}-1}L)-9M_{\max}(C^{k_{j}-2}L))(5M_{\max}(C^{k_{4}-1}L)+3M_{\max}(C^{k_{4}-1}))\nonumber.
\end{eqnarray}

\begin{equation*}
 \tiny{
 \setlength{\arraycolsep}{3pt}
L(C_{n}^{q})=\begin{pmatrix}
   \begin{array}{cccccc|c|cccc}
1 & -1 & 0 & \cdots & 0  & 0 & 0 & 0 & 0 & \cdots & 0\\
-1 & 2 & -1 & \cdots &  0 & 0  & 0 & 0 & \cdots & 0\\
0 & -1 & 2 & \cdots  & 0 & 0  & 0 & 0 & \cdots & 0\\
\vdots & \vdots & \vdots & \ddots & \vdots & \vdots & \vdots & \vdots & \vdots  & \vdots\\
0 & 0 & 0 & \cdots & 2 & -1  & 0 & 0 & 0 & \cdots & 0\\
0 & 0 & 0 & \cdots & -1 & 2  & -1 & 0 & 0 &\cdots & 0\\
\hline
0 & 0 & 0 & \cdots & 0  & -1 & n-k+1 & -1 &  -1 & \cdots & -1\\
\hline
0 & 0 & 0 & \cdots & 0 & 0 & -1 & 1 & 0 &\cdots & 0\\
\vdots & \vdots & \vdots & \vdots & \vdots &  \vdots & \vdots & \vdots & \ddots  & \vdots\\
0 & 0 & 0 & \cdots & 0 & 0  & 0 & 0 & \cdots & 1\\
\end{array}
 \end{pmatrix}
 }.
 \end{equation*}

(iii) When $n\geq 72$,
\begin{eqnarray}\label{equ34}
&& M_{\max}(T_{n}^{*})\nonumber\\
 &=&\prod_{j=0,3,4}(11M_{\max}(C^{k_{j}-1}L)-9M_{\max}(C^{k_{j}-2}L))\prod^{2}_{j=1}(13M_{\max}(C^{k_{j}}L)+6M_{\max}(C^{k_{j}}))\nonumber\\
&&+(13M_{\max}(C^{k_{1}}L)+6M_{\max}(C^{k_{1}}))[\prod_{j=0,2,3}(11M_{\max}(C^{k_{j}-1}L)-9M_{\max}(C^{k_{j}-2}L))\nonumber\\
&&\times(5M_{\max}(C^{k_{4}-1}L)+3M_{\max}(C^{k_{4}-1}))+\prod^{4}_{j=2}(11M_{\max}(C^{k_{j}-1}L)-9M_{\max}(C^{k_{j}-2}L))\nonumber\\
&&\times(5M_{\max}(C^{k_{0}-1}L)+3M_{\max}(C^{k_{0}-1}))+(13M_{\max}(C^{k_{2}}L)+6M_{\max}(C^{k_{2}}))\nonumber\\
&&\times[\prod_{j=0,1,4}(11M_{\max}(C^{k_{j}-1}L)-9M_{\max}(C^{k_{j}-2}L))(5M_{\max}(C^{k_{3}-1}L)+3M_{\max}(C^{k_{3}-1}))\nonumber\\
&&+\prod_{j=1,3,4}(11M_{\max}(C^{k_{j}-1}L)-9M_{\max}(C^{k_{j}-2}L))(5M_{\max}(C^{k_{0}-1}L)\nonumber\\
&&+3M_{\max}(C^{k_{0}-1}))]+\prod^{2}_{j=0}(11M_{\max}(C^{k_{j}-1}L)-9M_{\max}(C^{k_{j}-2}L))\prod_{j=3,4}(5\\
&&\times M_{\max}(C^{k_{j}-1}L)+3M_{\max}(C^{k_{j}-1}))+\prod_{j=0,3}(5M_{\max}(C^{k_{j}-1}L)+3M_{\max}(C^{k_{j}-1}))\nonumber\\
&&\times\prod_{j=1,2,4}(11M_{\max}(C^{k_{j}-1}L)-9M_{\max}(C^{k_{j}-2}L))+\prod_{j=0,4}(5M_{\max}(C^{k_{j}-1}L)\nonumber\\
&&+3M_{\max}(C^{k_{j}-1}))\prod^{3}_{j=1}(11M_{\max}(C^{k_{j}-1}L)-9M_{\max}(C^{k_{j}-2}L))+(5M_{\max}(C^{k_{0}-1})\nonumber\\
&&+3M_{\max}(C_{*}^{k_{0}-1}))\prod^{4}_{j=1}(11M_{\max}(C^{k_{j}-1}L)-9M_{\max}(C^{k_{j}-2}L)),\nonumber
\end{eqnarray}
where
\begin{equation*}
k_{0}=max\{0, \lfloor\frac{n-37}{35}\rfloor\},
k_{j}=
\begin{cases}
\lfloor\frac{n-2+7j}{35}\rfloor & \text{if $n\geq 37$};\\
\lfloor\frac{n-9+7j}{35}\rfloor & \text{if $n\leq 30$}.
\end{cases}
 \end{equation*}
 and $j\in \{1,2,3,4\}$.

(7) If $n\equiv 6(mod~7)$ and $n\geq 27, n\neq 34$, then

(i) $M_{\max}(T_{27}^{*})=5832$, $M_{\max}(T_{41}^{*})=597861$, $M_{\max}(T_{48}^{*})=6052320$, $M_{\max}(T_{55}^{*})=61268400$, $M_{\max}(T_{62}^{*})=620136000$, $M_{\max}(T_{69}^{*})=6276690000$.

(ii) When $n\geq 76$,
\begin{eqnarray}\label{equ35}
 &&M_{\max}(T_{n}^{*})\nonumber\\
 &=&(11M_{\max}(C^{k_{0}-1}L)-9M_{\max}(C^{k_{0}-2}L))\{\prod_{j=1,3,4,6}(3M_{\max}(C^{k_{j}})+6M_{\max}(C^{k_{j}-1}L))\nonumber\\
&&\times\prod_{j=2,5}(5M_{\max}(C^{k_{j}-1}F)+3M_{\max}(C^{k_{j}-1}FL))+(5M_{\max}(C^{k_{2}-1}F)\nonumber\\
&&+3M_{\max}(C^{k_{2}-1}FL))[\prod_{j=1,4,5,6}(3M_{\max}(C^{k_{j}})+6M_{\max}(C^{k_{j}-1}L))(5M_{\max}(C^{k_{3}-1}F)\nonumber\\
&&+3M_{\max}(C^{k_{3}-1}F-L))+\prod^{6}_{j=3}(3M_{\max}(C^{k_{j}})+6M_{\max}(C^{k_{j}-1}L))(5M_{\max}(C^{k_{1}-1}F)\nonumber\\
&&+3M_{\max}(C^{k_{1}-1}F-L))]+(5M_{\max}(C^{k_{5}-1}F)+3M_{\max}(C^{k_{5}-1}FL))\\
&&\times[\prod_{j=1,2,3,6}(3M_{\max}(C^{k_{j}})+6M_{\max}(C^{k_{j}-1}L))(5M_{\max}(C^{k_{4}-1}F)+3M_{\max}(C^{k_{4}-1}F\nonumber\\
&&-L))+\prod^{4}_{j=1}(3M_{\max}(C^{k_{j}})+6M_{\max}(C^{k_{j}-1}L))(5M_{\max}(C^{k_{6}-1}F)+3M_{\max}(C^{k_{6}-1}F \nonumber\\
&&-L))]+\prod_{j=1,2,5,6}(3M_{\max}(C^{k_{j}})+6M_{\max}(C^{k_{j}-1}L))\prod^{4}_{j=3}(5M_{\max}(C^{k_{j}-1}F)\nonumber\\
&&+3M_{\max}(C^{k_{j}-1}F-L))+\prod_{j=1,2,4,5}(3M_{\max}(C^{k_{j}})+6M_{\max}(C^{k_{j}-1}L))\nonumber\\
&&\prod_{j=3,6}(5M_{\max}(C^{k_{j}-1}F)+3M_{\max}(C^{k_{j}-1}F-L))+\prod_{j=2,3,5,6}(3M_{\max}(C^{k_{j}})\nonumber\\
&&+6M_{\max}(C^{k_{j}-1}L))\prod_{j=1,4}(5M_{\max}(C^{k_{j}-1}F)+3M_{\max}(C^{k_{j}-1}F-L))+\prod^{5}_{j=2}(3\nonumber\\
&&M_{\max}(C^{k_{j}})+6M_{\max}(C^{k_{j}-1}L))\prod_{j=1,6}(5M_{\max}(C^{k_{j}-1}F)+3M_{\max}(C^{k_{j}-1}F-L))\}\nonumber\\
&&+(5M_{\max}(C^{k_{0}-1}L)+3M_{\max}(C^{k_{0}-1}))[\prod_{j=1,2,3,4,6}(3M_{\max}(C^{k_{j}})+6M_{\max}(C^{k_{j}-1}L))\nonumber\\
&&\times(5M_{\max}(C^{k_{5}-1}F)+3M_{\max}(C^{k_{5}-1}FL))+\prod_{j=1,2,4,5,6}(3M_{\max}(C^{k_{j}})+6M_{\max}(C^{k_{j}-1}\nonumber\\
&&L))(5M_{\max}(C^{k_{3}-1}F)+3M_{\max}(C^{k_{3}-1}F-L))+\prod^{6}_{j=2}(3M_{\max}(C^{k_{j}})+6M_{\max}(C^{k_{j}-1}\nonumber\\
&&L))(5M_{\max}(C^{k_{1}-1}F)+3M_{\max}(C^{k_{1}-1}F-L))+\prod_{j=1,3,4,5,6}(3M_{\max}(C^{k_{j}})\nonumber\\
&&+6M_{\max}(C^{k_{j}-1}L))(5M_{\max}(C^{k_{2}-1}F)+3M_{\max}(C^{k_{2}-1}FL))+\prod_{j=1,2,3,5,6}(3M_{\max}(C^{k_{j}})\nonumber\\
&&+6M_{\max}(C^{k_{j}-1}L))(5M_{\max}(C^{k_{4}-1}F)+3M_{\max}(C^{k_{4}-1}F-L))+\prod^{5}_{j=1}(3M_{\max}(C^{k_{j}})\nonumber\\
&&+6M_{\max}(C^{k_{j}-1}L))(5M_{\max}(C^{k_{6}-1}F)+3M_{\max}(C^{k_{6}-1}F-L))]+\prod^{6}_{j=1}(3M_{\max}(C^{k_{j}})\nonumber\\
&&+6M_{\max}(C^{k_{j}-1}L))(5M_{\max}(C^{k_{0}-1})+3M_{\max}(C_{*}^{k_{0}-1}))\nonumber.
\end{eqnarray}
where $k_{j}=\lfloor\frac{n-27+7j}{49}\rfloor, j\in\{0,1,2\}$.
\end{theorem}

\begin{proof}
 According to the Gallai-Edmonds vertex partition, by formula (\ref{equ23}), Lemmas \ref{art31} and \ref{art33}, we can derive the results of (1), (2) and (3). In what follows, we proved  (4), (5) and (6).

(4) When $n\equiv 3(mod~7)$ and $n\geq 17$. By formula (\ref{equ23}), $M_{\max}(G_{17})=216, M_{\max}(G_{24})=2187$, $M_{\max}(G_{31})=22140, M_{\max}(G_{38})=224100$.

If $n\geq 45$, by Lemma \ref{art32}, a tree  $T^{1}$ and $C^{k_{0}}F$ are get from $T_{n}^{*}-uv$, and we have $C^{k_{0}}F-L$, $ C^{k_{1}}F$, $C^{k_{2}}F$, $C^{k_{3}}F$ from $T_{n}^{*}-u-v$. By Lemma \ref{art32} again, $C^{k_{1}}F$ and a tree denoted by $T^{2}$ are obtained from $T^{1}-u^{'}v^{'}$, and we have $C^{k_{1}}F-L$, $C^{k_{2}}F$, $C^{k_{3}}F$ from $T^{1}-u^{'}-v^{'}$. Then we can get the induced subgraphs $C^{k_{2}}F$, $C^{k_{3}}FL$ from $T^{2}-u^{''}v^{''}$, $C^{k_{2}}F-L$, $C^{k_{3}}F$ are derive from $T^{2}-u^{''}-v^{''}$. Hence
\begin{eqnarray*}
 &&M_{\max}(T_{n}^{*})\nonumber\\
 &=&M_{\max}(C^{k_{0}}F)\{{M_{\max}(C^{k_{1}}F)}[M_{\max}(C^{k_{2}}F) M_{\max}(C^{k_{3}}FL)\\
&& +M_{\max}(C^{k_{2}}F-L)M_{\max}(C^{k_{3}}F)]+M_{\max}(C^{k_{1}}F-L)M_{\max}(C^{k_{2}}F)\\
&&\times M_{\max}(C^{k_{3}}F)\}+M_{\max}(C^{k_{0}}F-L)M_{\max}(C^{k_{1}}F)M_{\max}(C^{k_{2}}F) \times M_{\max}(C^{k_{3}}F)\\
&=&\prod^{2}_{j=0}M_{\max}(C^{k_{j}}F) M_{\max}(C^{k_{3}}F L)+\prod_{j=0,1,3}M_{\max}(C^{k_{j}}F) M_{\max}(C^{k_{2}}F-L)\\
&&+\prod_{j=0,2,3}M_{\max}(C^{k_{j}}F) M_{\max}(C^{k_{1}}F-L)+\prod^{3}_{j=1}M_{\max}(C^{k_{j}}F) M_{\max}(C^{k_{0}}F-L).
\end{eqnarray*}
So by Lemma \ref{art31} and \ref{art33}, we have
\begin{eqnarray*}
&&M_{\max}(T_{n}^{*})\\
&=&\prod^{2}_{j=0}(3M_{\max}(C^{k_{j}})+6M_{\max}(C^{k_{j}-1}F))(5M_{\max}(C^{k_{3}-1}F))+3M_{\max}(C^{k_{3}-1}FL)\\
&&+\prod _{j=0,1,3}(3M_{\max}(C^{k_{j}})+6M_{\max}(C^{k_{j}-1}F)(5M_{\max}(C^{k_{2}-1}F)+3M_{\max}(C^{k_{2}-1}F-L))\\
&&+\prod _{j=0,2,3}(3M_{\max}(C^{k_{j}})+6M_{\max}(C^{k_{j}-1}F)(5M_{\max}(C^{k_{1}-1}F))+3M_{\max}(C^{k_{1}-1}F-L)\\
&&+\prod ^{3}_{j=1}(3M_{\max}(C^{k_{j}})+6M_{\max}(C^{k_{j}-1}F)(5M_{\max}(C^{k_{0}-1}F)+3M_{\max}(C^{k_{0}-1}F-L)),
\end{eqnarray*}
where $k_{j}=\lfloor \frac{n-17+7j}{28}\rfloor, j\in\{0,1,2,3\}$.

(5) If $n\equiv 5(mod~7)$ and $n\geq 12$, by formula (\ref{equ23}), we have $M_{\max}(G_{12})=41, M_{\max}(G_{19})=418$.
When $n\geq 26$, by Lemma \ref{art32}, analogously, we obtain $C^{k_{1}}L$ and a tree  $T^{3}$ from $T_{n}^{*}-uv$, and we get that $C^{k_{0}}L$, $C^{k_{1}}$, $C^{k_{2}}L$ and $L$ from $T_{n}^{*}-u-v$. Afterwards, we obtain $C^{k_{0}}P_{3}$, $C^{k_{2}}L$ from $T^{3}-u^{'}v^{'}$, and we have $C^{k_{0}}L$, $C^{k_{2}}$ from $T^{3}-u^{'}-v^{'}$. Hence
 \begin{eqnarray*}
 M_{\max}(T_{n}^{*})&=&M_{\max}(C^{k_{1}}L)[M_{\max}(C^{k_{0}} P_{3})M_{\max}(C^{k_{2}}L)+M_{\max}(C^{k_{0}}L)\\
&&\times M_{\max}(C^{k_{2}})] M_{\max}(C^{k_{0}}L)M_{\max}(C^{k_{1}})M_{\max}(C^{k_{2}}L)\\
&=&\prod^{2}_{j=1}M_{\max}(C^{k_{j}}L)M_{\max}(C^{k_{0}} P_{3})+\prod^{1}_{j=0}M_{\max}(C^{k_{j}}L)M_{\max}(C^{k_{2}})\\
&&+\prod_{j=0, 2}M_{\max}(C^{k_{j}}L) M_{\max}(C^{k_{1}}).
\end{eqnarray*}
Since $M_{\max}(L)=1$, we omit it in this paper. Therefore, by Lemmas \ref{art31} and \ref{art33} we have
 \begin{eqnarray*}
 M_{\max}(T_{n}^{*})&=&\prod^{2}_{j=1}(11M_{\max}(C^{k_{j}-1}L)-9M_{\max}(C^{k_{j}-2}L))(13M_{\max}(C^{k_{0}}L)+6M_{\max}(C^{k_{0}}))\\
 &&+\prod^{1}_{j=0}(11M_{\max}(C^{k_{j}-1}L)-9M_{\max}(C^{k_{j}-2}L))(5M_{\max}(C^{k_{2}-1}L)+3M_{\max}(C^{k_{2}-1}))\\ &&+\prod_{j=0,2}(11M_{\max}(C^{k_{j}-1}L)-9M_{\max}(C^{k_{j}-2}L))(5M_{\max}(C^{k_{1}-1}L)+3M_{\max}(C^{k_{1}-1})),
\end{eqnarray*}
where $k_{j}=\lfloor \frac{n-5+7j}{21}\rfloor, j\in\{0,1,2\}$.

(6) If $n\equiv 2(mod~7)$ and $n\geq 9$, by formula (\ref{equ23}), we have $M_{\max}(G_{9})=15, M_{\max}(G_{16})=153,\\ M_{\max}(G_{23})=1560, M_{\max}(G_{30})=15807$.

When $37\leq n \leq 65$, by Lemma \ref{art32}, the trees  $T^{4}$ and $T^{5}$ are obtained from $T_{n}^{*}-uv$, and we can obtain $C^{k_{1}}L$, $C^{k_{3}}L$ and a tree $T^{6}$ from $T_{n}^{*}-u-v$. Then we get $L$ and $T^{6}$ from $T^{5}-u^{'}v^{'}$, and we have $C^{k_{2}}L$, $C^{k_{4}}L$ from $T^{5}-u^{'}-v^{'}$. Since the structure of trees $T^{4}$, $T^{6}$ are the same as that of $T^{3}$, i.e. they have the same recurrences, so we have
\begin{eqnarray*}
 &&M_{\max}(T_{n}^{*})\nonumber\\
 &=&[M_{\max}(C^{k_{1}}P_{3}) M_{\max}(C^{k_{3}}L)+M_{\max}(C^{k_{1}}L) M_{\max}(C^{k_{3}})] [M_{\max}(C^{k_{2}} P_{3})\\
&&\times M_{\max}(C^{k_{4}}L)+M_{\max}(C^{k_{2}}L) M_{\max}(C^{k_{4}})+M_{\max}(C^{k_{2}}L) M_{\max}(C^{k_{4}}L)]\\
&&+M_{\max}(C^{k_{1}}L) M_{\max}(C^{k_{3}}L)[M_{\max}(C^{k_{2}}P_{3})M_{\max}(C^{k_{4}}L)+M_{\max}(C^{k_{2}}L) M_{\max}(C^{k_{4}})]\\
&=&\prod^{2}_{j=1}M_{\max}(C^{k_{j}} P_{3})\prod^{4}_{j=3}M_{m}(C^{k_{j}}L)+M_{\max}(C^{k_{1}} P_{3})[ \prod^{3}_{j=2}M_{\max}(C^{k_{j}}L)\\
&&\times M_{\max}(C^{k_{4}})+\prod^{4}_{j=2}M_{\max}(C^{k_{j}}L)]+ M_{\max}(C^{k_{2}}P_{3}) [\prod_{j=1,4}M_{\max}(C^{k_{j}}L) \\
&& \times M_{\max}(C^{k_{3}})+\prod_{j=1,3,4}M_{\max}(C^{k_{j}}L)]+\prod^{2}_{j=1}M_{\max}(C^{k_{j}}L)\prod^{4}_{j=3}M_{\max}(C^{k_{j}})\\
&&+\prod_{j=1,2,4}M_{\max}(C^{k_{j}}L) M_{\max}(C^{k_{3}})+ \prod^{3}_{j=1}M_{\max}(C^{k_{j}}L)M_{\max}(C^{k_{4}}).
\end{eqnarray*}
And then we can obtain  formula (\ref{equ33}) by Lemmas \ref{art31} and \ref{art33}.

When $n\geq72$, by Lemma \ref{art32}, we can get trees $T^{7}$ and $T^{8}$ from $T_{n}^{*}-uv$, and we have $C^{k_{1}}L$, $C^{k_{3}}L$ and a tree  $T^{9}$ from $T_{n}^{*}-u-v$. Next, $C^{k_{0}}$ and a tree $T^{10}$ are derived from $T^{9}-u^{'}v^{'}$, and we obtain $C_{*}^{k_{0}}$, $C^{k_{2}}L$, $C^{k_{4}}L$ from $T^{9}-u^{'}-v^{'}$. Owing to the structures of trees $T^{7}$ and $T^{10}$ are the same as the structure of $T^{3}$, and the structure of the tree $T^{8}$ is also the same as $T_{n}^{*}$ when $n\equiv 5(mod~7)$, i.e. they have the same recurrences. Consequently,
\begin{eqnarray*}
 &&M_{\max}(T_{n}^{*})\nonumber\\
 &=&[M_{\max}(C^{k_{1}} P_{3}) M_{\max}(C^{k_{3}}L)+M_{\max}(C^{k_{1}}L) M_{\max}(C^{k_{3}})] \{M_{\max}(C^{k_{0}}L)\\
&&\times [M_{\max}(C^{k_{2}}P_{3})M_{\max}(C^{k_{4}}L)+M_{\max}(C^{k_{2}}L) M_{\max}(C^{k_{4}})]+M_{\max}(C^{k_{0}})\\
&&\times M_{\max}(C^{k_{2}}L) M_{\max}(C^{k_{4}}L)\}+M_{\max}(C^{k_{1}}L)M_{\max}(C^{k_{3}}L)\{M_{\max}(C^{k_{0}})\\
&&\times [M_{\max}(C^{k_{2}}P_{3})M_{\max}(C^{k_{4}}L)+M_{\max}(C^{k_{2}}L)M_{\max}(C^{k_{4}})]+M_{\max}(C^{k_{0}}_{*})\\
&&\times M_{\max}(C^{k_{2}}L) M_{\max}(C^{k_{4}}L)\}\\
&=&\prod_{j=0,3,4}M_{\max}(C^{k_{j}}L)\prod_{j=1,2}M_{\max}(C^{k_{j}}P_{3})+M_{\max}(C^{k_{1}} P_{3})[\prod_{j=0,2,3}M_{\max}(C^{k_{j}}L)\\
&&\times M_{\max}(C^{k_{4}})+\prod^{4}_{j=2}M_{\max}(C^{k_{j}}L)M_{\max}(C^{k_{0}})]+M_{\max}(C^{k_{2}} P_{3})[\prod_{j=0,1,4}M_{\max}(C^{k_{j}}L)\\
&&\times M_{\max}(C^{k_{3}})+\prod_{j=1,3,4}M_{\max}(C^{k_{j}}L)M_{\max}(C^{k_{0}})]+\prod^{2}_{j=0}M_{\max}(C^{k_{j}}L)\prod^{4}_{j=3}M_{\max}(C^{k_{j}})\\
&&+\prod_{j=0,3}M_{\max}(C^{k_{j}})\prod_{j=1,2,4}M_{\max}(C^{k_{j}}L)+\prod_{j=0,4}M_{\max}(C^{k_{j}})\prod^{3}_{j=1}M_{\max}(C^{k_{j}}L)\\
&&+M_{\max}(C^{k_{0}}_{*})\prod^{4}_{j=1}M_{\max}(C^{k_{j}}L).\\
\end{eqnarray*}
So we have the result of formula (\ref{equ34}) by Lemma \ref{art31} and \ref{art33}, where
\begin{equation*}
k_{0}=max\{0, \lfloor\frac{n-37}{35}\rfloor\},
k_{j}=
\begin{cases}
\lfloor\frac{n-2+7j}{35}\rfloor & \text{if $n\geq 37$};\\
\lfloor\frac{n-9+7j}{35}\rfloor & \text{if $n\leq 30$}.
\end{cases}
 \end{equation*}
 and $j\in \{1,2,3,4\}$.

(7) If $n\equiv 6(mod~7)$ and $n\geq 27, n\neq 34$, by formula (\ref{equ23}), we have $M_{\max}(T_{27}^{*})=5832$, $M_{\max}(T_{41}^{*})=597861$, $M_{\max}(T_{48}^{*})=6052320$, $M_{\max}(T_{55}^{*})=61268400$, $M_{\max}(T_{62}^{*})=620136000$, $M_{\max}(T_{69}^{*})=6276690000$. When $n\geq 76$, by Lemma \ref{art32}, we have trees  $T^{11}$ and $T^{12}$ from $T_{n}^{*}-uv$, and we get $C^{k_{1}}F$, $C^{k_{3}}F$, $C^{k_{5}}F$ and the tree $T^{13}$ from $T_{n}^{*}-u-v$. The $C^{k_{0}}L$ and $T^{14}$ are obtained from $T^{12}-u^{'}v^{'}$, and we can obtain $C^{k_{0}}$, $C^{k_{2}}F$, $C^{k_{4}}F$, $C^{k_{6}}F$ from $T^{12}-u^{'}-v^{'}$. Then we have $C^{k_{0}}$ and $T^{14}$ from $T^{13}-u^{''}v^{''}$, the subgraphs $C^{k_{0}}_{*}$, $C^{k_{2}}F$, $C^{k_{4}}F$, $C^{k_{6}}F$ are derived from $T^{13}-u^{''}-v^{''}$. Since $T^{11}, T^{14}$ have the same structure of $T^{1}$, i.e. they have the same recurrences, we have
\begin{eqnarray*}
 &&M_{\max}(T_{n}^{*})\nonumber\\
 &=&[M_{\max}(C^{k_{1}}F)(M_{\max}(C^{k_{3}}F)M_{\max}(C^{k_{5}}F L)+M_{\max}(C^{k_{3}}F-L)M_{\max}(C^{k_{5}}F))\\
&& +M_{\max}(C^{k_{1}}F-1)M_{\max}(C^{k_{3}}F)M_{\max}(C^{k_{5}}F)]\{M_{\max}(C^{k_{0}}L)[M_{\max}(C^{k_{6}}F)\\
&&\times(M_{\max}(C^{k_{2}}FL) M_{m}(C^{k_{4}}F)+M_{\max}(C^{k_{2}}F)M_{\max}(C^{k_{4}}F-L))+M_{\max}(C^{k_{2}}F)\\
&&\times M_{\max}(C^{k_{4}}F)M_{\max}(C^{k_{6}}F-1)]+M_{\max}(C^{k_{0}})M_{\max}(C^{k_{2}}F)M_{\max}(C^{k_{4}}F)\\
&&\times M_{\max}(C^{k_{6}}F)\}+M_{\max}(C^{k_{1}}F)M_{\max}(C^{k_{3}}F)M_{\max}(C^{k_{5}}F)\{M_{\max}(C^{k_{0}})\\
&&\times [M_{\max}(C^{k_{6}}F)(M_{\max}(C^{k_{2}}FL)M_{\max}(C^{k_{4}}F)+M_{\max}(C^{k_{2}}F)M_{\max}(C^{k_{4}}F-L))\\
&&+M_{\max}(C^{k_{2}}F)M_{\max}(C^{k_{4}}F)M_{\max}(C^{k_{6}}F-1)]+M_{\max}(C^{k_{0}}_{*})M_{\max}(C^{k_{2}}F)\\
&&\times M_{\max}(C^{k_{4}}F)M_{\max}(C^{k_{6}}F)\}\\
&=&M_{\max}(C^{k_{0}}L)\{\prod_{j=1,3,4,6}M_{\max}(C^{k_{j}}F)\prod_{j=2,5}M_{\max}(C^{k_{j}}FL)+M_{\max}(C^{k_{2}}F L)\\
&&\times[\prod_{j=1,4,5,6}M_{\max}(C^{k_{j}}F)M_{\max}(C^{k_{3}}F-L)+\prod^{6}_{j=3}M_{\max}(C^{k_{j}}F)M_{\max}(C^{k_{1}}F-L)]\\
&&+M_{\max}(C^{k_{5}}F L)[\prod_{j=1,2,3,6}M_{\max}(C^{k_{j}}F)M_{\max}(C^{k_{4}}F-L)+\prod^{4}_{j=1}M_{\max}(C^{k_{j}}F)\\
&&\times M_{\max}(C^{k_{6}}F-L)]+\prod_{j=1,2,5,6}M_{\max}(C^{k_{j}}F)\prod^{4}_{j=3}M_{\max}(C^{k_{j}}F-L)\\
&&+\prod_{j=1,2,4,5} M_{\max}(C^{k_{j}}F)\prod_{j=3,6}M_{\max}(C^{k_{j}}F-L)+\prod_{j=2,3,5,6}M_{\max}(C^{k_{j}}F)\\
&&\times\prod_{j=1,4}M_{\max}(C^{k_{j}}F-L)+\prod^{5}_{j=2}M_{\max}(C^{k_{j}}F)\prod_{j=1,6}M_{\max}(C^{k_{j}}F-L)\}+M_{\max}(C^{k_{0}})\\
&&\times[\prod_{j=1,2,3,4,6}M_{\max}(C^{k_{j}}F)M_{\max}(C^{k_{5}}FL)+\prod_{j=1,2,4,5,6}M_{\max}(C^{k_{j}}F)M_{\max}(C^{k_{3}}F-L)\\
&&+\prod^{6}_{j=2}M_{\max}(C^{k_{j}}F)M_{\max}(C^{k_{1}}F-L)+\prod_{j=1,3,4,5,6}M_{\max}(C^{k_{j}}F)M_{\max}(C^{k_{2}}F L)\\
&&+\prod_{j=1,2,3,5,6}M_{\max}(C^{k_{j}}F)M_{\max}(C^{k_{4}}F-L)+\prod^{5}_{j=1}M_{\max}(C^{k_{j}}F)M_{\max}(C^{k_{6}}F-L)]\\
&&+\prod^{6}_{j=1}M_{\max}(C^{k_{j}}F)M_{\max}(C^{k_{0}}_{*}).
\end{eqnarray*}
Hence, by Lemmas \ref{art31} and \ref{art33}, we derive the formula (\ref{equ35}), where $k_{j}=\lfloor\frac{n-27+7j}{49}\rfloor, j\in\{0,1,2\}$.
\end{proof}

Besides, if $n\in\{6,10,13,20,34\}$, by formula (\ref{equ23}), it is easy to get that $M_{\max}(T_{1}^{*})=M_{\max}(T_{2}^{*})=M_{\max}(T_{3}^{*})=1, M_{\max}(T_{6,1}^{*})=M_{\max}(T_{6,2}^{*})=5, M_{\max}(T_{10}^{*})=21, M_{\max}(T_{13}^{*})=56, M_{\max}(T_{20}^{*})=571, M_{\max}(T_{34,1}^{*})=M_{\max}(T_{34,2}^{*})=59049$.

Next, we will give two examples.

{\bf Example: }If the tree with $n=72$, it satisfies the condition $n\equiv 2(mod~7)$ and $n\geq 9$, by calculating we have $k_{0}=1, k_{1}=k_{2}=k_{3}=k_{4}=2$. See Figure \ref{fig4}(a).
\begin{figure}[htbp]
   \centering
   \subfigure[$n=72$]{
       \begin{minipage}[t]{0.3\linewidth}
          \centering
          \includegraphics[scale=0.5]{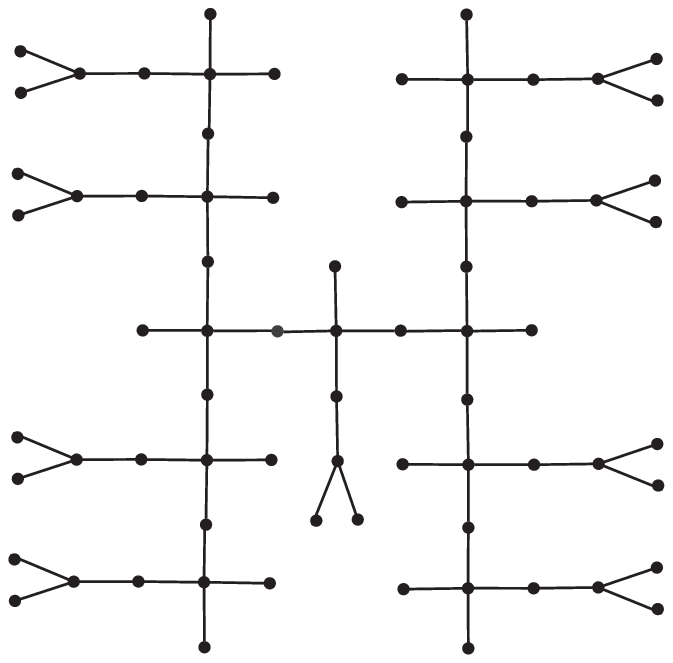}\\

     \end{minipage}
     }
     \subfigure[$n=76$]{
       \begin{minipage}[t]{0.33\linewidth}
          \centering
          \includegraphics[scale=0.5]{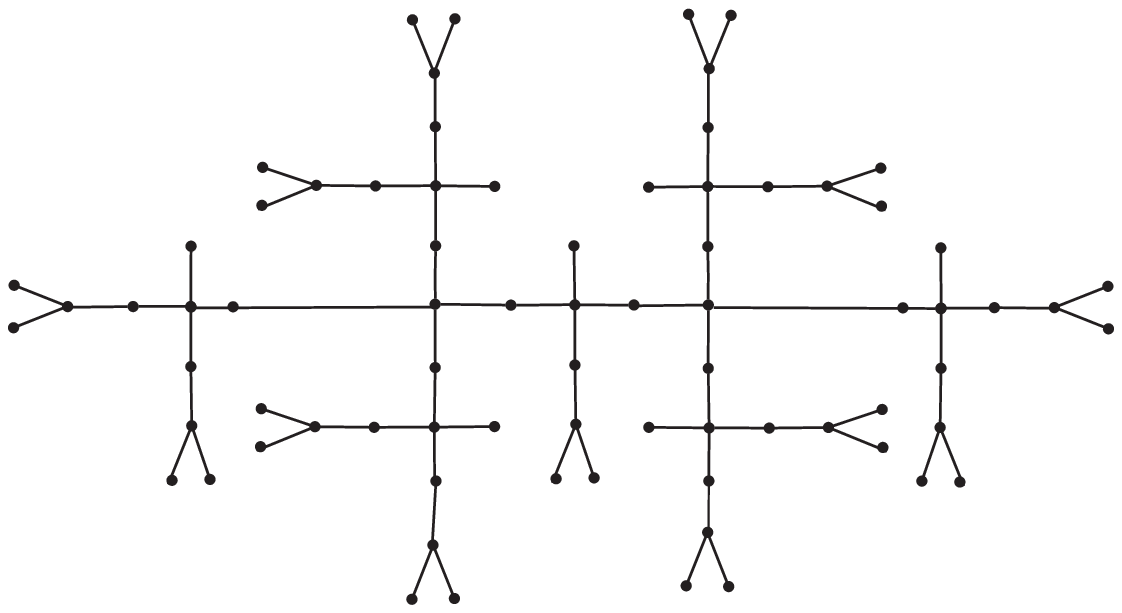}\\

     \end{minipage}
     }
     \centering
     \caption{The structure of two trees.}
     \vspace{-0.2cm}
     \label{fig4}
\end{figure}
 By Theorem \ref{art34} we have
\begin{eqnarray*}
M_{\max}(T_{72}^{*})&=&11\times112^{2}\times191^{2}+[191\times(11\times112^{2}\times79+112^{2}\times8)]\times2+11\times112^{2}\times79^{2}\\
&&+(8\times79\times112^{3})\times2+5\times112^{4}=16915082240.
\end{eqnarray*}
When the tree with order $n=76$, clearly, it satisfies the condition $n\equiv 6(mod~7)$ and $n\geq 27, n\neq 34$, so $k_{0}=k_{1}=k_{2}=k_{3}=k_{4}=k_{5}=k_{6}=1$. See Figure \ref{fig4}(b). By Theorem \ref{art34} we can obtain that
\begin{eqnarray*}
M_{\max}(T_{76}^{*}))&=&11\times [(30^{4}\times 21^{2})\times5+21\times(30^{4}\times21)\times2]+8\times(30^{4}\times21^{2})\times6+5\times30^{6}\\
&=&63503190000.
\end{eqnarray*}

\noindent{\bf Conflicts of Interest}

 The authors declare no conflict of interest.

\noindent{\bf Founding}\\
\indent This research is supported by the National Natural Science
Foundation of China (No. 12261071), the Natural Science Foundation of Qinghai Province (No. 2020-ZJ-920).


\end{document}